\documentclass[12pt]{article}
\usepackage{amssymb,epsfig,amsfonts,epic,graphics,pictex, 
} 

\begin{document}

\newtheorem{lem}{Lemma}[section]
\newtheorem{prop}{Proposition}
\newtheorem{con}{Construction}[section]
\newtheorem{defi}{Definition}[section]
\newtheorem{coro}{Corollary}[section]
\newcommand{\hf}{\hat{f}}
\newtheorem{fact}{Fact}[section]
\newtheorem{theo}{Theorem}
\newcommand{\Br}{\Poin}
\newcommand{\Cr}{{\bf Cr}}
\newcommand{\dist}{{\rm dist}}
\newcommand{\diam}{\mbox{diam}\, }
\newcommand{\mod}{{\rm mod}\,}
\newcommand{\compose}{\circ}
\newcommand{\dbar}{\bar{\partial}}
\newcommand{\Def}[1]{{{\em #1}}}
\newcommand{\dx}[1]{\frac{\partial #1}{\partial x}}
\newcommand{\dy}[1]{\frac{\partial #1}{\partial y}}
\newcommand{\Res}[2]{{#1}\raisebox{-.4ex}{$\left|\,_{#2}\right.$}}
\newcommand{\sgn}{{\rm sgn}}

\newcommand{\CC}{\mathbb{C}}
\newcommand{\D}{{\bf D}}
\newcommand{\Dm}{{\bf D_-}}
\newcommand{\RR}{\mathbb{R}}
\newcommand{\NN}{\mathbb{N}}
\newcommand{\HH}{\mathbb{H}}
\newcommand{\ZZ}{\mathbb{Z}}

\newcommand{\tr}{\mbox{Tr}\,}
\newcommand{\R}{{\bf R}}
\newcommand{\C}{{\bf C}}

\newenvironment{nproof}[1]{\trivlist\item[\hskip \labelsep{\bf Proof{#1}.}]}
{\begin{flushright} $\square$\end{flushright}\endtrivlist}
\newenvironment{proof}{\begin{nproof}{}}{\end{nproof}}

\newenvironment{block}[1]{\trivlist\item[\hskip \labelsep{{#1}.}]}{\endtrivlist}
\newenvironment{definition}{\begin{block}{\bf Definition}}{\end{block}}

\newtheorem{conjec}{Conjecture}

\newtheorem{com}{Comment}
\font\mathfonta=msam10 at 11pt
\font\mathfontb=msbm10 at 11pt
\def\Bbb#1{\mbox{\mathfontb #1}}
\def\lesssim{\mbox{\mathfonta.}}
\def\suppset{\mbox{\mathfonta{c}}}
\def\subbset{\mbox{\mathfonta{b}}}
\def\grtsim{\mbox{\mathfonta\&}}
\def\gtrsim{\mbox{\mathfonta\&}}

\newcommand{\Poin}{{\bf Poin}}
\newcommand{\Bo}{\Box^{n}_{i}}
\newcommand{\Di}{{\cal D}}
\newcommand{\gd}{{\underline \gamma}}
\newcommand{\gu}{{\underline g }}
\newcommand{\ce}{\mbox{III}}
\newcommand{\be}{\mbox{II}}
\newcommand{\F}{\cal{F}}
\newcommand{\Ci}{\bf{C}}
\newcommand{\ai}{\mbox{I}}
\newcommand{\dupap}{\partial^{+}}
\newcommand{\dm}{\partial^{-}}
\newenvironment{note}{\begin{sc}{\bf Note}}{\end{sc}}
\newenvironment{notes}{\begin{sc}{\bf Notes}\ \par\begin{enumerate}}%
{\end{enumerate}\end{sc}}
\newenvironment{sol}
{{\bf Solution:}\newline}{\begin{flushright}
{\bf QED}\end{flushright}}

\title{Limit Drift}

\author{Genadi Levin
\thanks{Supported in part by an ISF grant 1378/13}\\
\small{Einstein Institute of Mathematics}\\
\small{Hebrew University}\\
\small{Givat Ram 91904, Jerusalem, ISRAEL}\\
\small{\tt levin@math.huji.ac.il}\\
\and
Grzegorz \'{S}wia\c\negthinspace tek
\thanks{Supported in part by a grant 2012/05/B/ST1/00551 funded by Narodowe Centrum Nauki.}\\
\small{Department. of Mathematics and Information Science}\\
\small{Politechnika Warszawska}\\
\small{Koszykowa 75}\\
\small{00-662 Warszawa, POLAND}\\
\small{\tt g.swiatek@mini.pw.edu.pl}
}
\normalsize
\maketitle

\abstract{We study the problem of the existence of wild attractors for critical circle 
coverings with Fibonacci dynamics. This is known to be related to the drift for the corresponding 
fixed points of renormalization. The fixed point depends only on the order of the critical point 
$\ell$ and its drift is a number $\vartheta(\ell)$ which is finite for each finite $\ell$. 
We show that the limit $\vartheta(\infty):=\lim_{\ell\rightarrow\infty} \vartheta(\ell)$ exists and is 
finite. The finiteness of the limit is in a sharp contrast with the case of Fibonacci unimodal maps. 
Furthermore, $\vartheta(\infty)$ is expressed as a contour integral in
terms of the limit of the fixed 
points of renormalization when $\ell\rightarrow\infty$. There is a
certain paradox here, since this dynamical limit  is a circle homemomorphism
with the golden mean rotation number whose own drift is $\infty$ for
topological reasons. } 

\section{Introduction.}
Recall that a wild attractor is a compact forward invariant set $K$ 
of zero Lebesgue measure such that for each open neighborhood $U$ 
of $K$ every orbit from
a subset of $U$ with positive Lebesgue measure stays forever in $U$ 
and tends to $K$,
while every orbit from a residual subset of $U$ is outside a certain open neighborhood of $K$ 
infinitely often  under forward
iterations. The existence of wild attractors for one-dimensional 
smooth maps was first
established in~\cite{bkns} for Fibonacci unimodal maps of the interval. 
Wild attractors appear when the order $\ell$ of 
the critical point becomes sufficiently large. That work introduced 
a general probabilistic framework 
for finding a wild attractor once a certain Markov structure has been 
found in the map. 
That work was followed by a program of
S. Van Strien and T. Nowicki for showing the existence of a similar
attractor for a complex polynomial, which would imply that the Julia
set of such a polynomial has positive measure. While that program has
not been followed to successful completion, partial progress was made based
on choosing sufficiently high criticality and studying limits when it
tended to $\infty$, see~\cite{ns} (cf.~\cite{buffcher} for complex maps with a positive
area Julia set).

Another class of real maps in which a wild attractor was expected was a critical 
covering map of the circle 
of Fibonacci type. 
Note that this problem is equivalent to an apparently different 
problem of the existence of non-escaping ``fat'' Cantor sets:
do there exist two branches 
of increasing $C^2$-maps 
of two disjoint subinterval of the unit interval $I$ onto $I$, so that 
the set of 
non-escaping points is a Cantor set of positive length? 
Necessarily, one of the branches must have
a critical point of inflection type and we assume that the itinerary
of this point has the the Fibonacci combinatorics.


In this paper, we study a certain quantity called the {\em drift} which is closely related to 
the problem of existence of a wild attractor for the fixed point map of 
renormalizations of the Fibonacci critical circle covers. As we prove in~\cite{lsuniv},
such a fixed point map $H_\ell$ exists and is unique for every odd $\ell\ge 3$
where $\ell$ is the order at a critical point. 

The concept of the drift can be tracked down to~\cite{bkns}, though the exact definition varies 
depending on the case under consideration and techniques used. 
In our case, we introduce a number $\vartheta(\ell)$ (called drift) such that it is positive  
if and only if the wild attractor exists. 
The key difference of our approach with~\cite{bkns} is that we introduce 
the drift not with respect
of the Lebesgue measure, but with respect to
an invariant measure of some induced map with infinite number of branches. 
This is crucial for further analysis.
Explicitly, a similar definition of the drift appeared in~\cite{morsmania} in the context
of unimodal maps.

We prove that the drift $\vartheta(\ell)$ tends 
to a limit which is expressed exclusively through the limit 
map with infinite criticality
(although as noted the sign of this limit would determine the
existence of wild attractor for maps with big, but {\it finite} criticality).
In the main step of the proof, we represent the drift $\vartheta(\ell)$ as an integral.
This is a completely new step, which allows us to deform
contour of integration to the complex plane
and analyze its behavior as the order $\ell$ tends to infinity.

One should note that the meaning of the limiting drift
is not related to the problem of existence of a wild attractor for the 
limit map
of infinite criticality itself-that question is easy because the limit map 
is a critical circle homeomorphism.

Of course, the key question here is what is the sign of the limit drift.
This question is much simpler to answer 
if the limit is infinite as this is the case for the
Fibonacci unimodal maps, and, in fact, could be extracted from~\cite{bkns}. 
Details are provided in section~\ref{sfuni}.
As we show for the covering maps, the limit is finite. 
This explains why the problem of wild attractor for unimodal maps
is simpler than for critical covering maps.

\paragraph{The class of dynamics.} The following maps were studied in~\cite{lsuniv}.

Consider open intervals $I^0, I^{-1}, I$ in the following configuration: 
$\overline{I}^0 \cap \overline{I}^{-1} = \emptyset$, $\overline{I}^0 \cup \overline{I}^{-1} \subset I$,
$0\in I^0$. 
The map $g$ is defined by its branches $\psi^0, \psi^{-1}$ on the corresponding intervals, which are 
both monotone increasing, $C^3$, and map onto $I$. Furthermore, each branch $\psi^i$, $i=0,-1$,
has exactly one critical point with the local representation as 
$(\xi^i(x))^{\ell}$ where $\xi^i$ are diffeomorphisms and $\ell$ an odd integer bigger then $1$. 
This implies that both branches have the same critical value at $0$. 
The set of such mappings is called ${\cal G}_{\ell}$.    

Maps from ${\cal G}_{\ell}$ can be obtained naturally by inducing from critical circle coverings, 
see~\cite{lsuniv}. 
However, to use the results of that, and most other 
papers, one should consider branches $f(x) = \sqrt[\ell]{g(x^{\ell})}$. 
The effect of this change of normalization is that there is only one critical point at $x=0$.

A map $g$ from ${\cal G}_{\ell}$ is said  
to have the Fibonacci combinatorics if there exists a weakly order preserving map from the circle to $I$ with endpoints 
identified, which conjugates $g$ on the forward orbit 
of the critical value to the dynamics of  an orbit under the golden mean rotation.

Suppose that $g$ is a map from ${\cal G}_{\ell}$ with Fibonacci combinatorics. 
We normalize $g$ so that $g(0)=1$. 
Let $g_1$ be the first return map into $I^0$ with the domains restricted only to those points which return after $1$ or $2$ iterates
It is clear that the new map $g_1$ still belongs to ${\cal G}_{\ell}$ with the new range which is $I^0$. 
Let $\zeta$ be a linear rescaling which maps $g_1(0)$ to $1$. 
Then, the renormalization of $g$, ${\cal R}(g) = \zeta \circ g_1 \circ \zeta^{-1}$.

From Theorem 1 in~\cite{lsuniv} we get the following:

\begin{fact}\label{fa:13xp,1}
For each $\ell$ which is an odd integer greater than $1$, there exists exactly one map $H_{\ell} \in {\cal G}_{\ell}$ 
with the Fibonacci combinatorics and constant 
$\tau =\tau_\ell < -1$ so that if $\phi=\phi^0, \phi^{-1}$ are the branches of $H_{\ell}$, then 
\begin{itemize}
\item $\phi(0) = 1$,
\item \[    \phi^{-1}(x) = \tau \phi \tau^{-1}(x)  \] 
for  $x\in I^{-1}$, 
\item the {\em fixed point equation} holds for all  $x \in \tau^{-1}I^0$:
\begin{equation}\label{equ:27up,1} 
\tau^{-1}\phi \tau(x) = \tau \phi \tau^{-1} \phi(x)\; .
\end{equation}
\item
for any $g\in {\cal G}_{\ell}$ with the Fibonacci combinatorics, 
the sequence of renormalizations ${\cal R}^n(g)$ converges to $H_{\ell}$ 
uniformly on the domain of $H_{\ell}$. 
\end{itemize} 
\end{fact}

\paragraph{The drift.}
The branch $\phi: I^0\to I$ of $H_\ell$ extends
to a real-analytic homeomorphism $\phi_\ell: (\tau_\ell^2 X_\ell, \tau_\ell x_{0,\ell})\to (-\infty, \tau_\ell^2)$ 
where $X_\ell\in I^0$, $X_\ell<0$, is such that $\phi_\ell(X_\ell)=\tau_\ell X_\ell$
and $x_{0,\ell}\in (\tau_\ell^2 X_\ell, X_\ell)$ is a single critical point of $\phi_\ell$
(see Section~\ref{1} for more details).
Denote $J=(\tau_\ell^2 X_\ell, X_\ell)$.
The tower (cf.~\cite{profesorus},~\cite{lsuniv})
$T : \cup_{n\in \ZZ} \tau_\ell^{-n}J \to \RR$ can be defined 
where $T|_{\tau_\ell^{-n}J}=\tau_\ell^{-n}\phi_\ell \tau_\ell^n$. 
One can be see from~(\ref{equ:27up,1}) that 
$$T|_{\tau_\ell^{-n}(J\cap I^0)}=H_\ell^{S_n}, \ \ \ n=0,1,2,...,$$
where $S_0=1$, $S_1=2$, $S_{n}=S_{n-1}+S_{n-2}$, $n\ge 2$, is the Fibonacci sequence. 

The {\em induced dynamics} $\Phi_{\ell} :\:  J \rightarrow J$  is defined with the exception of
countably many points and in the form $\Phi_{\ell}(x) = \tau^{-m(x)}_{\ell} \phi_{\ell}(x)$
where $m(x)=m(x,\ell)$ is the unique integer for which 
$\tau^{-m(x)}_{\ell} \phi_{\ell}(x) \in J$.
In other words, $m(x)=-k(\phi_\ell(x))$ where $k(x)=n$ for $x\in \tau_\ell^{-n}J$ and $n\in \ZZ$. 
Note that the {\em drift function} $m(x)$ is
unbounded: 
$$\inf_{x\in J} m(x)=-\infty, \ \ \ \sup_{x\in J}m(x) =\infty,$$
and $m\in L_p(J)$ for every $p>0$.

The map $\Phi_{\ell}: J\to J$
has a unique invariant probability measure $\mu_\ell$ on $J$ which 
is absolutely continuous with respect to the Lebesgue measure, see Section~\ref{1}. The {\em drift} is 
$$\vartheta(\ell) :=  -\int_{J} m(x) d\mu_\ell(x).$$ 
Importance of this number comes from the fact that the sign of $\vartheta(\ell)$ is responsible
for the existence of wild attractors for the map $H_\ell$. 
Indeed, since $\Phi_{\ell}$ is ergodic, by the Birkhoff Ergodic Theorem,
for Lebesgue almost every point $x\in J$,
\begin{equation}\label{average} 
-\lim_{n\to \infty} \frac{1}{n}\sum_{i=0}^{n-1} m(\Phi_{\ell}^i(x))=\vartheta(\ell).
\end{equation}
On the other hand, 
$$-\sum_{i=0}^n m(\Phi_{\ell}^i(x))=k(T^{n}(x)), \ \ n=0,1,... .$$
It follows from this and~(\ref{average}) that, if $\vartheta(\ell)<0$, then Lebesgue almost every point  
leaves the domain of $H_\ell$ under the tower dynamics. 
In particular, $H_\ell$ has no wild attractor in this case.
If $\vartheta(\ell)>0$, by ~(\ref{average}), 
for Lebesgue almost every point $x$, $T^n(x)\to 0$ as $n\to \infty$.
It follows that there is a set of positive Lebesgue measure 
$E$ in the domain of $H_\ell$ such that for every 
$x\in E$ all forward iterates $T^n(x)$, $n\ge 0$, do not leave this domain and $T^n(x)\to 0$.
Every such $x$ must be in the basin of attraction of the orbit of $0$ of the map $H_\ell$. 
The proof uses the fact that $H_\ell^{S_n}(0)=\tau_\ell^{-n}\in \tau_\ell^{-n}J$, $n\ge 0$, and repeats the one
for the unimodal maps, see~\cite{bruin}, Lemma 5.1.
Therefore, the set $E$ belongs to the basin of attraction of the orbit of $0$.

The case $\vartheta(\ell)=0$ can be handled with help of the Central Limit Theorem for the map $\Phi_{\ell}$
(it is proved under some assumptions on the map and the drift function
which hold for $\Phi_{\ell}$ and $m$). It follows, a.e. point must leave 
the range of $H_\ell$, hence, $H_\ell$ has no wild attractor.

Note also that presumably the fixed point map $H_\ell$ has a wild attractor if and only if 
a wild attractor exists for all maps of the class ${\cal G}_{\ell}$
(see~\cite{morsmania}, where a similar claim is proved in the class of Fibonacci unimodal maps).

\paragraph{Acknowledgement.}
The second author acknowledges that the work on the problem of the drift 
started as a joint project with Edson Vargas.

\section{Induced Dynamics and Invariant Densities}\label{1}
The following notations are consistent with those used
in~\cite{leswi:common}, see also Introduction.

By $\ell$ we will denote an odd integer greater than $1$ or
$\infty$. For every $\ell$ we have the parameters
$\tau_{\ell},X_{\ell},x_{0,\ell}$ subject to the following inequalities:
\[ \tau_{\ell} < \min(X_{\ell}\tau^2_{\ell},-1) \leq \tau_{\ell}^2
X_{\ell} \leq x_{0,\ell} < X_{\ell} < 0 < \tau_{\ell}X_{\ell} < 1 \; .\]
For $\ell$ finite the inequality $\tau^2_{\ell} X_{\ell} < x_{0,\ell}$ is
sharp, but for $\ell=\infty$ these points are equal. 

As finite $\ell$ tend to $\infty$, sequences
$(\tau_{\ell}),(X_{\ell}),(x_{\ell,0})$ tend to limits equal to their
respective values for $\ell=\infty$, see Theorem 1, p. 698,
in~\cite{leswi:common}.    

Also, for every $\ell$ a mapping $\phi_{\ell} :\:
(\tau^2_{\ell}X_{\ell},\tau_{\ell}x_{0,\ell}) \rightarrow \RR$ is given,
which is strictly increasing. For every $\ell$ we have 
\[ \begin{array}{ccc}
\lim_{x\searrow x_{0,\ell}} \phi_{\ell}(x) & = & 0\\
\phi_{\ell}(X_{\ell}) & = & \tau_{\ell}X_{\ell} \\
\phi_{\ell}(0) & = & 1 \\
\lim_{x\nearrow \tau_{\ell}x_{0,\ell}} \phi_{\ell}(x) & = & \tau_{\ell}^2
\end{array} \; .\]
The derivative of $\phi_{\ell}(x)$ is strictly positive whenever
$x_{0,\ell} < x < \tau_{\ell} x_{0,\ell}$. 

Additional information is available for $\ell<\infty$. 
\begin{fact}\label{fact:24ga,1}
For every $\ell<\infty$, the representation
\[ \phi_{\ell}(x) = \left( E(x) \right)^{\ell} \]
holds, where $E:(\tau_{\ell}^2X_{\ell},\tau_{\ell}x_{0,\ell}) \rightarrow
(-\infty,\sqrt[\ell]{\tau^2_{\ell}})$ is a diffeomorphism from the
Epstein class.   
\end{fact}

For a proof see Lemma 2.1, p. 703, of~\cite{leswi:common}. In
particular, for $\ell$ finite the range of $\phi_{\ell}$ is
$(-\infty,\tau_{\ell}^2)$, but for $\ell=\infty$ it is only
$(0,\tau_{\infty}^2)$.  

\paragraph{Convergence statement.}
From~\cite{leswi:common} we get the following.
\begin{fact}\label{fact:24gp,1}
Transformations $(\log(\phi_{\ell}))^{-1}$ with inverse branches chosen so
that 
\[ (\log(\phi_{\ell}))^{-1}\circ\log(\phi_{\ell}(x))=x \; .\]
for $x\in(x_{0,\ell},\tau_{\ell}x_{0,\ell})$ converge to $(\log(\phi_{\infty}))^{-1}$
almost uniformly on 
\[ U:=\CC\setminus[\log\tau^2_{\infty},\infty)\; .\]  
This should be understood to include the claim that on every compact subset    
of $U$ mappings $(\log(\phi_{\ell}))^{-1}$ are defined for $\ell$ large
enough. 
\end{fact}

Theorem 1 of~\cite{leswi:common} gives convergence on a non-trivial
segment of the real line. Extensions of $\phi_{\ell}^{-1}$ to $\CC\setminus
[\tau_{\ell}^2,\infty)$  exist by Fact~\ref{fact:24ga,1} and form a
  normal family by Montel's theorem leading to our statement. 

\subsection{Induced dynamics and its properties.}
For each $\ell$ consider its {\em inducing domain} $(\tau^2_{\ell}
X_{\ell},X_{\ell})$. The image of this interval by $\phi_{\ell}$ is 
$(-\infty,\tau_{\ell} X_{\ell})$ for $\ell$ finite and
$(0,\tau^2_{\infty})$ for $\ell=\infty$. In any case, with the
exception of countably many points it decomposes into a disjoint sum
of $\tau^m_{\ell}(\tau^2_{\ell} X_{\ell}, X_{\ell})$  for various
integer values of $m$. 

\begin{defi}\label{defi:24ga,1}
The {\em induced dynamics} 
\[ \Phi_{\ell} :\:  (\tau^2_{\ell}X_{\ell},X_{\ell}) \rightarrow
(\tau^2_{\ell}X_{\ell},X_{\ell})\]  is defined with the exception of
countably many points and in the form
\[ \Phi_{\ell}(x) = \tau^{-m(x,\ell)}_{\ell} \phi_{\ell}(x) \]
where $m(x,\ell)$ is the unique integer for which 
$\tau^{-m(x,\ell)}_{\ell} \phi_{\ell}(x) \in (\tau^2_{\ell}X_{\ell},X_{\ell})$. 
\end{defi}

We will talk about inverse branches of $\Phi_{\ell}$ written as
\begin{equation}\label{equ:24ga,1}
 \psi_{m,\ell}(x) = \phi^{-1}_{\ell}(\tau^m_{\ell} x) \; .
\end{equation}
The {\em domain} of $\psi_{m,\ell}$ means the largest interval on
which formula~(\ref{equ:24ga,1}) defines a diffeomorphism.


For some combinations at $\ell$ and $m$ the domain of $\psi_{m,\ell}$
may be empty. It is non-empty if $m$ is a negative odd integer, or $m$
is any even integer and $\ell<\infty$. 

The basic property for us is:
\begin{lem}\label{lem:24ga,1}
For every $\ell$ and $m$ the domain of $\psi_{m,\ell}$
 if non-empty, contains at
least the interval $(\tau^3_{\ell},0)$. 
\end{lem} 
\begin{proof}
For the domain of $\psi_{m,\ell}$ to be non-empty, the set
$\phi_{\ell}(\tau_{\ell}^2X_{\ell},X_{\ell})$ must contain the segment 
$\tau_{\ell}^{m}(\tau^2_{\ell} X_{\ell},X_{\ell})$.

If this segment is to the right of $0$, then $m$ is an odd negative
integer. Then $\phi^{-1}_{\ell}$ defines a diffeomorphism on $(0,\tau^2_{\ell})$.
After rescaling by $\tau^{-m}_{\ell}$ this gives $(\tau_{\ell}^{-m+2},0)$ as
the domain, which contains at least $(\tau^3_{\ell},0)$.  

If $\tau_{\ell}^{m}(\tau^2_{\ell} X_{\ell},X_{\ell})$ is to the left
of $0$, which may only happen for $\ell<\infty$,
then $\phi_{\ell}^{-1}$ can be used to define a diffeomorphism on the
entire negative-half line onto $(\tau^{\ell}X_{\ell},x_{0,\ell})$ which
is invariant by the rescaling, so it is the domain of $\psi_{m,\ell}$.
\end{proof}

\begin{lem}\label{lem:24gp,1}
Let $p_0$ be a point chosen so that  
\[ \tau_{\ell} < p_0 < X_{\ell}\tau^2_{\ell} \]
for all $\ell$ sufficiently large or infinite. Then for every $m$ for which the domain of
$\psi_{m,\ell}$ is non-empty, $\psi_{m,\ell}((p_0,0])\subset
\left( \tau_{\ell}^2,\psi_{-1,\ell}(p_0) \right)$. 
\end{lem}
\begin{proof}
Let us first consider the case of $m$ negative and odd. Then
$\tau_{\ell}^{m}(p_0,0) \subset (0,p_0\tau_{\ell}^{-1})$. A further mapping by
  $\phi^{-1}_{\ell}$ brings this to 
\[ \left( x_{\ell,0},\phi_{\ell}^{-1}(p_0\tau_{\ell}^{-1})\right) =
\left( x_{\ell,0},\psi_{-1,\ell}(p_0) \right) \]
so the claim holds in this case. 

For $m$ even the range of every $\psi_{\ell,m}$ is contained in
$(\tau_{\ell}^2X_{\ell},x_{0,\ell})$.
\end{proof}

\begin{lem}\label{lem:24gp,2}
There is a fixed closed segment $I_0$ contained in the interior of
$(p_0,0)$ such that for every $m$ and $\ell$ sufficiently large or infinite
\[ \psi_{m,\ell}(p_0,0) \subset I_0 \; .\]
\end{lem}   
\begin{proof} 
This segment must be contained in
$(\tau^2_{\ell}X_{\ell},\psi_{-1,\ell}(p_0))$ for all $\ell$. But the
the convergence statement, see Fact~\ref{fact:24gp,1}, 
\[ q_0:=\lim_{\ell\rightarrow\infty}\psi_{-1,\ell}(p_0) =
\psi_{-1,\infty}(p_0) < \psi_{-1,\infty}(\tau_{\infty}) = 0 \]
where the inequalities come from the fact that $\psi_{-1,\ell}$, or
generally $\psi_{m,\ell}$ for $m$ odd, are order-reversing. Then as
$I_0$ we can pick any segment with the right endpoint in
$(p_0,x_{0,\ell})$ and the left endpoint in $(q_0,0)$. 
\end{proof}  

\begin{coro}\label{coro:24gp,1}
For all $\ell$ sufficiently large or infinite $\Phi_{\ell}$ is
uniformly expanding in the Poincar\'{e} metric of $\CC\setminus\{
x\in\RR :\: x\notin [p_0,0]\}$.
\end{coro}

\subsection{Perron-Frobenius-Ruelle transfer operators.}
The transfer operator is given by the formula
\begin{equation}\label{equ:24gp,1}
P_{\ell} f(z) = \sum_{m} (-1)^m \psi'_{m,\ell}(z) f\circ \psi_{m,\ell}(z) \; .
\end{equation}

The definition is, of course, not complete, without specifying the
space to which $f$ belong. 

We will use the term {\em density} to denote a real-valued
non-negative function whose integral over the indicated space with
respect to the Lebesgue measure is $1$. 

\begin{fact}\label{fact:24gp,2}
For every $\ell$ sufficiently large or infinite, the Perron-Frobenius
operator acting on the space $L_1(\tau_{\ell}^2 X_{\ell}, X_{\ell})$
has unique non-negative invariant density $f_{\ell} =
\lim_{n\rightarrow\infty} P_{\ell}^n {\bf e}$, which in fact is 
positive, where ${\bf e}$ is a constant density.  
\end{fact}
This follows from the Folklore Theorem in the light of
Corollary~\ref{coro:24gp,1}.  

\paragraph{Functional spaces.}
Let $D_0$ denote the disk whose diameter is the segment $I_0$ from
Lemma~\ref{lem:24gp,2}. We will also need a closed segment $I_1$ such
that 
\[ I_0 \subset  \mbox{int}\:(I_1) \subset I_1 \subset (p_0,0) \; .\]
Then $D_1$ denotes the disk whose diameter is $I_1$.  

The functional space ${\cal F}_0$ will consist
of all functions bounded and holomorphic on $D_0$, real on the real
line, as well continuous
on $\overline{D_0}$, with the $C^0$
(supremum) norm. It is clearly a Banach space. We will also use a
space ${\cal F}_1$ of all continuous complex-valued functions, real on
the real line, defined on
the disk $D_1$ with the $C^0$ norm. 


\paragraph{Properties of the transfer operator.}
Looking at formula~(\ref{equ:24gp,1}) one can see that the
transfer operator is well-defined on ${\cal F}_1$ as a formal
sum, since $\psi_{m,\ell}(D_1)\subset D_0$ by Lemma~\ref{lem:24gp,1}
and the Epstein class property expressed in Fact~\ref{fact:24ga,1}. 

To consider convergence, observe the following simple fact:
\begin{lem}\label{lem:24gp,3} 
There exists a constant $K_1$ such that for all positive integers $n$,
all sequences $m_1,\cdots,m_n$, all $\ell$
sufficiently large or infinite all $z\in D_1$ 
\[ |(\tilde{\psi})'(z)| \leq K_1
|\tilde{\psi}(\tau^2_{\ell}X_{\ell},X_{\ell})|\; \]
where $\tilde{\psi} = \psi_{m_n,\ell}\circ\cdots\circ \psi_{m_1,\ell}$.  
\end{lem}
\begin{proof}
From Lemma~\ref{lem:24ga,1} and Fact~\ref{fact:24ga,1} all
$\tilde{\psi}$ have univalent continuations to $\CC\setminus\{
x\in\RR :\: x\notin [\tau_{\ell}^3,0]\}$ so by Koebe's lemma 
\[ |(\tilde{\psi})'(z)| \leq L_1  |(\tilde{\psi})'(\hat{z})| \]
for any $\hat{z}\in D_1$, where the constant $L_1$ is independent of
$n, (m_n), \ell, z$. But for a suitably chosen real $\hat{z}$ 
\[ |(\tilde{\psi})'(\hat{z})| =
\frac{|\tilde{\psi}(\tau^2_{\ell}X_{\ell},X_{\ell})|}{|(\tau^2_{\ell}X_{\ell},X_{\ell})|} \]
    which yields the lemma, since the denominator is bounded below by 
\[ 2^{-1}|(x_{0,\infty},\tau^{-2}_{\infty}x_{0,\infty})|\] 
 for all $\ell$
    sufficiently large.  
 \end{proof}

It is now clear that the sum in formula~\ref{equ:24gp,1} converges
uniformly in the operator norm on ${\cal F}_1$ and so the operator is
well-defined. Furthermore,
\begin{lem}\label{lem:24gp,4}
For $\ell$ sufficiently large or infinite operator $P_{\ell}$ acting
on ${\cal F}_0$ is compact. 
\end{lem}
\begin{proof}
For any $f$ from the unit ball in ${\cal F}_0$,
formula~(\ref{equ:24gp,1}) 
shows that $P_{\ell}f$ is a holomorphic element of ${\cal F}_1$. From
Lemma~\ref{lem:24gp,3} its norm is bounded by $K_1 |(\tau_{\ell}^2
X_{\ell},X_{\ell})|$. So, $P_{\ell} B_{{\cal F}_0}(0,1)$  is a normal
family on $D_1$. Then, after
restriction to ${\overline D}_{0}$ every sequence from this family has
a uniformly convergent subsequence, but this means compactness in ${\cal
  F}_0$. 
\end{proof}

Additionally we get
\begin{lem}\label{lem:25ga,1}
There is a constant $K_2$ such that for the transfer operator
acting on ${\cal F}_{0}$ for all $\ell$ and $n$ 
$\| P^n_{\ell} \| \leq K_2$.
\end{lem}
\begin{proof}
The transfer operator for iterates is given by 
\[ P^n_{\ell} f(z) = \sum_{m_1,\cdots,m_n} (-1)^{\sum
  m_k}\left(\psi_{m_n,\ell}\circ\cdots\circ\psi_{m_1,\ell}\right)'(z)\cdot 
\left( f\circ\psi_{m_n,\ell}\circ\cdots\circ\psi_{m_1,\ell}(z) \right) \; .\]
By Lemma~\ref{lem:24gp,3}, 
\[ \| P^n_{\ell} f \| \leq K_1 \|f\| \sum_{m_1,\cdots,m_n} \left|
\psi_{m_n,\ell}\circ\cdots,\circ\psi_{m_1,\ell}(\tau_{\ell}^2X_{\ell},X_{\ell})
\right| \; .\] 

Segments $\psi_{m_n,\ell}\circ\cdots,\circ\psi_{m_1,\ell}(\tau_{\ell}^2X_{\ell},X_{\ell})$
are domains of branches of $\Phi_{\ell}^n$, hence disjoint and
contained in $(\tau_{\ell}^2 X_{\ell},X_{\ell})$. Their sum is
uniformly bounded by $|p_0|$ and so we get the claim. 
\end{proof}

\paragraph{Convergence of densities.}
We are ready to prove our main result for this section, namely that 
\begin{prop}\label{prop:24gp,1}
As $\ell\rightarrow\infty$ invariant densities $f_{\ell}$ tend to
$f_{\infty}$ almost uniformly on
$(x_{0,\infty},\tau^{-2}_{\infty}x_{0,\infty})$. 
\end{prop}

\begin{lem}\label{lem:24gp,5}
For every $f\in {\cal F}_0$
\[ \lim_{\ell\rightarrow\infty} P_{\ell} f = P_{\infty} f \; .\]
\end{lem}
\begin{proof}
Since the sum in~(\ref{equ:24gp,1}) converges in the operator norm, it
is sufficient to show that for each term $\psi'_{m,\ell}
f\circ\psi_{m,\ell}$ converge in ${\cal F}_0$. This is so since
$\psi_{m,\ell}$ converge to $\psi_{m,\infty}$ analytically, hence in
$C^1$, by Fact~\ref{fact:24gp,1}.
\end{proof}

\begin{lem}\label{lem:25ga,2}
For every $\ell$ sufficiently large or infinite, there is $f_{0,\ell}
\in{\cal F}_0$ which is a fixed point of $P_{\ell}$ and when restricted
to $(\tau_{\ell}^2X_{\ell},X_{\ell})$ it is equal to $f_{\ell}$
introduced by Fact~\ref{fact:24gp,2}.
\end{lem}
\begin{proof}
Consider the constant density ${\bf e}$ for which 
\[ P^n_{\ell}{\bf e} \rightarrow f_{\ell}\]
 in $L_1(\tau_{\ell}^2X_{\ell},X_{\ell})$. In
${\cal F}_0$ the sequence $(P^n_{\ell} {\bf e})$ is bounded by 
Lemma~\ref{lem:25ga,1} and so belongs to a compact set by
Lemma~\ref{lem:24gp,4}. Let $f_{0,\ell}$ be an accumulation point of
this sequence. On the segment $(\tau_{\ell}^2 X_{\ell},X_{\ell})$ it is
equal to $f_{\ell}$ in $L_1$. So any two accumulation points are
simply equal on this segment, and so equal in ${\cal F}_0$ as
holomorphic functions. It develops that $P^n_{\ell} {\bf e}
\rightarrow f_{0,\ell}$ and so $f_{0,\ell}$ is a fixed point of
$P_{\ell}$. 
\end{proof}

\paragraph{Proof of Proposition~\ref{prop:24gp,1}.}
The sequence $(f_{0,\ell})=(P_{\ell} f_{0,\ell})$ has a convergent subsequence in  
${\cal F}_0$. Let us call its limit $\tilde{f}$. Observe that
$P_{\infty}\tilde{f} = \tilde{f}$. Indeed, for any $\ell$ sufficiently large, 
\[ \| P_{\infty}\tilde{f}-\tilde{f} \| \leq \|
P_{\infty}\tilde{f}-P_{\ell}\tilde{f} \| + \|P_{\ell}\tilde{f} -
P_{\ell} f_{0,\ell} \| + \|P_{\ell} f_{0,\ell} - \tilde{f}\| \; .\] 
As $\ell$ tend to $\infty$ all three terms tend to $0$. For the first one
this follows by Lemma~\ref{lem:24gp,5}. The second one can be
estimated by $K_2 \|\tilde{f}-f_{0,\ell}\|$ by Lemma~\ref{lem:25ga,1}
and the final one is just $\|\tilde{f}-f_{0,\ell}\|$ again since
$f_{0,\ell}$ is a fixed point of $P_{\ell}$. 

But now $\tilde{f}$ restricted to $I_0$ is a real-valued function and
a uniform limit of similarly restricted $f_{0,\ell}$. By
Lemma~\ref{lem:25ga,2} $f_{0,\ell}$ is a density on
$(\tau_{\ell}^2X_{\ell},X_{\ell})$. It follows that $\tilde{f}$ is a
density on $(x_{0,\infty},\tau^{-2}_{\infty}
x_{0,\infty})$. Furthermore, since $\tilde{f}$ was a fixed point of
$P$ and all $\psi_{m,\infty}$ map $(x_{0,\infty},\tau^{-2}_{\infty}
x_{0,\infty})$ into itself, this density is invariant. By the uniqueness claim in
fact~\ref{fact:24gp,2} it is equal to $f_{0,\infty}$. Thus,
$\tilde{f}$ and $f_{0,\infty}$ obtained from Lemma~\ref{lem:25ga,2}
are equal on $(x_{0,\infty},\tau^{-2}_{\infty}
x_{0,\infty})$ and hence equal in ${\cal F}_0$. 

Recall that $\tilde{f}$ was the limit in ${\cal F}_0$ of an arbitrary
convergent subsequence of $(f_{0,\ell})$. Hence $f_{0,\ell}
\rightarrow f_{0,\infty}$ in ${\cal F}_0$ which implies almost uniform
convergence of $f_{\ell}$ to $f_{\infty}$ on $(x_{0,\infty},\tau^{-2}_{\infty}
x_{0,\infty})$. 

This ends the proof of Proposition~\ref{prop:24gp,1}. 
\section{The Drift}
\begin{defi}\label{defi:25ga,1}
Recall Definition~\ref{defi:24ga,1} of the induced dynamics $\Phi$ and
in particular the function $m(x,\ell)$. For any $\ell<\infty$ the {\em drift} is 
\[ \vartheta(\ell) :=  -\int_{\tau^2_{\ell}X_{\ell}}^{X_{\ell}} m(x,\ell) f_{\ell}(x)\;
dx \; .\]
\end{defi} 

Recall fact~\ref{fact:24gp,1} and the set $U$ defined there. Also, we
get a holomorphic function $f_{0,\infty}\in {\cal F}_0$ from
Lemma~\ref{lem:25ga,2}. $f_{0,\infty}$ is defined on the disk $D_0$ 
which contains $(x_{0,\infty},\tau^{-2}_{\infty}x_{0,\infty})$ in its
interior.

Here is the main theorem of this paper:
\begin{theo}\label{theo:25ga,1}
There is $\theta :\: [0,1) \rightarrow (\log\phi_{\infty})^{-1}(U) \cap D_0$, an arc which is piecewise smooth on every compact subset of $[0,1)$, and satisfies
\begin{enumerate}
\item
$\theta(0)=\tau_{\infty}^{-2}x_{0,\infty}$, $\lim_{t\rightarrow 1}\theta(1) = x_{0,\infty}$ and $\Im \theta(t) > 0$ for all $t\in(0,1)$,
\item
the image of $\theta$ and the segment
$(x_{0,\infty},\tau^{-2}_{\infty}x_{0,\infty}]$ are contained in
  simply connected open subset of $(\log\phi_{\infty})^{-1}(U) \cap D_0$,
\item
\[ \lim\sup_{t\rightarrow 1} \frac{|\Re \theta(t) - \Re x_{0,\infty}|}{|\theta(t)-x_{0,\infty}|^3\log|\theta(t)-x_{0,\infty}|^{-1}}
< \infty \; .\]
\end{enumerate}

Additionally, the formula
\[ \lim_{\ell\rightarrow\infty} \vartheta(\ell) = (\log|\tau_{\infty}|)^{-1}\lim_{\epsilon\rightarrow 0}
\;\Re\left[ \int_0^{1-\epsilon}
  \log\frac{\phi_{\infty}(\theta(t))}{\theta(t)} f_{0,\infty}(\theta(t)) \frac{d\theta}{dt} \,
dt \right]\; \]
holds. 
\end{theo}

\begin{coro}\label{coro:5hp,1}
Let $W$ be any open neighborhood of
$[x_{0,\infty},\tau^{-2}_{\infty}x_{0,\infty}$. When $\epsilon_0$ is small enough, the arc
$[x_{0,\infty},\tau^{-2}_{\infty}x_{0,\infty}]+\theta$ is a Jordan
curve which is homotopically trivial in $(\log\phi_{\infty})^{-1}(U)
\cap W$. 
\end{coro}
\begin{proof}
Clearly, the closed curve mentioned here is a Jordan curve by
Theorem~\ref{theo:25ga,1}.  For $\epsilon_0$ small enough, the bounded
connected component of its complement is contained in $W$, but then
it is also contained in $(\log\phi_{\infty})^{-1}(U)$ which is simply
connected.
\end{proof}

\paragraph{Integrated density as a coordinate.}
Since $f_{0,\ell}$ is a positive density on
$(x_{0,\ell},\tau^{-2}_{\ell}x_{0,\ell})$, 
\[\mu(x) = \int_{x_{0,\ell}}^x f_{0,\infty}(z)\, dz\]

defines a univalent map from a neighborhood of the interval
$[x_{0,\ell},\tau^{-2}_{\infty} x_{0,\ell}]$ contained in $D_0$ onto a neighborhood of
$[0,1]$. This coordinate not only normalizes that inducing domain in a
way independent of $\ell$, but also internalizes that invariant
density factor in integration. 

Fix a neighborhood $W'$ of $[0,1]$. Then for $\ell$ large enough or
infinite, $\mu^{-1}_{\ell}$ is defined on $W'$. Let $W'_{\ell} =
\mu^{-1}_{\ell}(W_0)$.  Since $\mu_{\ell}$ is defined on $W'_{\ell}$,
we must have $W'_{\ell} \subset D_0$. 

\begin{defi}\label{defi:5hp,1}
Let $\ell$ be finite and large enough or infinite, as specified
above. 
Introduce a function $\hat{\phi}_{\ell} := \phi_{\ell}\circ
\mu_{\ell}^{-1}$. Also, define set $W_{\ell}$ to be the connected
component of the intersection of $\mu_{\ell}\circ
(\log\phi_{\ell}^{-1}(U) \cap W'$. 
\end{defi}

Surely, $W_{\ell}$ is contained in the domain of
$\hat{\phi}_{\ell}$. It is also simply connected, because connected
components of the intersection of two simply connected sets
$W'_{\ell}$ and $(\log\phi_{\ell})^{-1}(U)$ are simply connected. 

Theorem~\ref{theo:25ga,1} preserves its meaning in the new coordinate 
and the main integral formula becomes

\begin{equation}\label{equ:5hp,1}
 \lim_{\ell\rightarrow\infty} \vartheta(\ell) = \lim_{\epsilon\rightarrow 0}
\; (\log|\tau_{\infty}|)^{-1}\Re\left[ \int_0^{1-\epsilon}
  \log\frac{\hat{\phi}_{\infty}(\theta(t))}{\mu^{-1}(\theta(t))} \frac{d\theta}{dt} \,
dt \right] \; .
\end{equation}

\subsection{A few identities.}
The proof of Theorem~\ref{theo:25ga,1} is based among other things on
a several identities which will be given now. 

\paragraph{Associated dynamics.}
For every $\ell$ we define {\em associated maps} 
\begin{equation}\label{equ:26ga,2}
\begin{array}{ccc}
G_{\ell} & = & \tau_{\ell}^{-1} \phi_{\ell} \tau_{\ell}^{-1} \\
\Gamma_{\ell} & = & \tau_{\ell} \phi_{\ell} \tau_{\ell}^{-2} \\
\end{array}
\end{equation}

Their relevant properties will now be listed.
\begin{fact}\label{fact:26ga,2}
\begin{itemize}
\item
For each $\ell$, $G_{\ell}=\Gamma_{\ell}\circ\Gamma_{\ell}$ and $\Gamma_{\ell}$ fixes $\tau^2_{\ell}X_{\ell}$
  with a negative derivative. Also, the functional equation 
\[ \phi_{\ell} \circ G_{\ell} = \tau^{-2}_{\ell} \phi_{\ell} \]
holds. 
\item
For any finite $\ell$ point $x_{0,\ell}$ is an attracting fixed point
of $G_{\ell}$ and $\tau^2_{\ell} X_{\ell}$ is a repelling fixed point.
\item
For $\ell=\infty$, $x_{0,\ell}$ is a neutral point for $G_{\infty}$
with the expansion 
\[ G_{\infty}(z) = z - c (z-x_{0,\infty})^3 + O(|z|^4) \]
for $c>0$. 
\end{itemize}
\end{fact}
The identities follow from the functional equation                     
$\phi_{\ell}=\tau_\ell^2 \phi_{\ell} \tau_\ell^{-1} \phi_{\ell}  \tau_\ell^{-1}$. 
The other claims are contained in Theorem 3 and Lemma 2.1 of~\cite{leswi:common}.

We will use the associated dynamics in the integrated density
coordinate, for example 
\begin{equation}\label{equ:26gp,3} {\hat \Gamma_{\ell}}
:=\mu_{\ell}\circ\Gamma_{\ell}\circ\mu_{\ell}^{-1}\; .
\end{equation}

Claims of Fact~\ref{fact:26ga,2} have obvious reformulations in the new
coordinate. We will just quote one:
\begin{equation}\label{equ:26ga,3}
\hat{\phi}_{\ell}\circ\hat{G}_{\ell} = \tau_{\ell}^{-2}\hat{\phi}_{\ell} 
\end{equation}

where $\hat{\phi}_{\ell}$ is introduced by Definition~\ref{defi:5hp,1}.  
 
\begin{lem}\label{lem:26ga,2}
For every finite $\ell$ there exist positive constants $K(\ell)$ and
$\epsilon(\ell)$ such that 
\[ \begin{array}{ccc}
\forall |z-x_{0,\ell}|<\epsilon(\ell)\; & |\log \hat{\phi}_{\ell}(z)| & \leq K(\ell) \log |z-x_{0,\ell}|^{-1}\\
\forall |z-\tau^2_{\ell}X_{\ell}|<\epsilon(\ell)\; & |\log
\hat{\phi}_{\ell}(z)| & \leq K(\ell) \log
|z-\tau^2_{\ell}X_{\ell}|^{-1} \; .
\end{array} \]
\end{lem}
The branch of the logarithm in $\log\hat{\phi}_{\ell}$ is chosen so
that the value if real on $(0,1)$.  
\begin{proof}
Both $x_{0,\ell}$ and $\tau_{\ell}^2 X_{\ell}$ are hyperbolic fixed
points for $\hat{G}_{\ell}$. For example, for $z$ in a neighborhood of $\tau^2_{\ell}
X_{\ell}$,  
\[ \log\hat{\phi}_{\ell}(z) = n \log\tau^2_{\ell} + \log\hat{\phi}_{\ell}(
  \hat{G}^n_{\ell}(z)) \; .\]
Choosing $G^n_{\ell}(z)$ in some fixed ring of outer radius $r$ and picking $1<\rho$ so
that it is smaller than $|\hat{G}_{\ell}'(z)|$ inside this ring, we get
\[ | \log\phi_{\ell}(z) | \leq   \frac{\log r/|z|}{\log
  \rho}\log\tau^2_{\ell} + L(\ell) \]
where $L(\ell)$ is a estimate on the ring.  The claim of the Lemma
follows and we proceed in the same way near $x_{0,\ell}$ using
$\hat{G}^{-1}_{\ell}$ instead of $\hat{G}_{\ell}$. 
\end{proof}

As a corollary:
\begin{coro}\label{coro:26ga,1}
For any $\ell<\infty$ function
$\log\hat{\phi}_{\ell}$ is integrable on any piecewise Lipschitz arc
in a neighborhood of $[0,1]$ intersected with the domain of
$\log\hat{\phi}_{\ell}$. 
\end{coro}

\paragraph{Formula for the drift.}

\begin{lem}\label{lem:26ga,1}
For every finite $\ell$ 
\[ \vartheta(\ell) = -(\log|\tau_{\ell}|)^{-1}\Re \int_0^1
\log\frac{\hat{\phi}_{\ell}(x)}{\mu^{-1}_{\ell}(x)}\, dx\; .    \]
\end{lem}
\begin{proof}

Since $d\mu_{\ell}(x)$ is an invariant measure of the induced map
$\Phi_{\ell}$, see Definition~\ref{defi:24ga,1}, for any
complex-valued integrable function $f$ one gets

\[ \int_{\tau_{\ell}^2X_{\ell}}^{X_{\ell}} f\circ\Phi_{\ell}(x)\,
d\mu_{\ell}(x) = \int_{\tau_{\ell}^2X_{\ell}}^{X_{\ell}} f(x)\,
d\mu_{\ell}(x) \; .\]
Picking $f(x) = \log |x|$ and recalling 
$\Phi_{\ell}(x) = \tau^{-m_{\ell}(x)}\phi_{\ell}(x)$ leads to 
\[  \int_{\tau_{\ell}^2X_{\ell}}^{X_{\ell}} \left[ \log|\phi_{\ell}(x)|
  - \log|\tau_{\ell}| m(x) - \log |x|\right]\, d\mu_{\ell}(x) = 0\]
which results in 
\[   \log|\tau_{\ell}| \int_{\tau_{\ell}^2X_{\ell}}^{X_{\ell}}
m_{\ell}(x)\, d\mu_{\ell}(x) = \int_{\tau_{\ell}^2X_{\ell}}^{X_{\ell}}
\log\frac{|\phi_{\ell}(x)|}{|x|}\, d\mu_{\ell}(x)\; . \]
Given the Definition~\ref{defi:25ga,1} the left-hand side is just 
$ -\log |\tau_{\ell}|\vartheta(\ell)$. The right-hand side can be
expressed as
\[ \Re \int_0^1 \log \frac{\hat{\phi}_{\ell}(x)}{\mu_\ell^{-1}(x)}\, dx\]
in the integrated density coordinate which leads the formula of~(\ref{lem:26ga,1}). 
\end{proof}

\paragraph{An identity for the invariant density $f_{\ell}$.}
Observe that for every $\ell$ the mapping $\tau^{-2}\Gamma_{\ell}$ transforms the
interval $(\tau^2_{\ell}X_{\ell}, X_{\ell})$ onto a smaller interval.
Since $\tau_\ell^2 X_\ell$ is fixed by $\Gamma_\ell$ by Fact~\ref{fact:26ga,2},
$\tau_\ell^{-2}\Gamma_\ell(\tau_\ell^2 X_\ell)=X_\ell$. Then, since $\Gamma_\ell$ reverses the orientation,
$\tau_\ell^{-2}\Gamma_\ell(X_\ell)<X_\ell$. Also, $\tau_\ell^{-2}\Gamma_\ell(X_\ell)=
\tau_\ell^{-1}\phi_\ell(\tau_\ell^{-2}X_\ell)>\tau_\ell^{-1}\phi_\ell(\tau_\ell X_\ell)=\tau_\ell^2 X_\ell$.
\begin{lem}\label{lem:26gp,1}
There exists $\epsilon_3>0$ such that for all $\ell$ sufficiently
large and finite and
$x\in (\tau^2_{\ell}x_{\ell},\tau^2_{\ell}X_{\ell}+\epsilon_3)$ one gets
\[ f_{\ell}(\tau^{-2}_{\ell}\Gamma_{\ell}(x)) (\tau^{-2}_{\ell}\Gamma_{\ell})'(x) -
f_{\ell}(\Gamma_{\ell}(x))\Gamma'_{\ell}(x) = \tau_{\ell}^{-2}
f_{\ell}(\tau^{-2}_{\ell}(x))  \; .\]
\end{lem}
\begin{proof}
We have $P_{\ell} f_{\ell} = f_{\ell}$ which expands to 
\[ f_{\ell}(x) = \sum_{m} (-1)^m (\psi'_{m,\ell}(x))
f_{\ell}(\psi_{m,\ell}(x)) \]
by formula~(\ref{equ:24gp,1}). So, for any differentiable
function $g(x)$
\[ f_{\ell}(g(x)) g'(x) = \sum_{m} (-1)^m (\psi_{m,\ell}\circ g)'(x) 
f_{\ell}(\psi_{m,\ell}\circ g)(x) \; .\]
We see that 
\[ \psi_{m,\ell}\circ(\tau^{-2}_{\ell}\Gamma_{\ell}) = \psi_{m-2,\ell}
\circ\Gamma_{\ell} \]
by formula~(\ref{equ:24ga,1}). So, 
\[   f_{\ell}(\tau^{-2}_{\ell}\Gamma_{\ell}(x))
(\tau^{-2}_{\ell}\Gamma_{\ell})'(x) = \sum_m (-1)^m
(\psi_{m-2,\ell}\circ\Gamma_{\ell})'(x)
f_{\ell}((\psi_{m-2,\ell}\circ\Gamma_{\ell})(x))\; .\]

Naturally, $f_{\ell}(\Gamma_{\ell}(x))\Gamma'_{\ell}(x)$ is the same sum, but with
$m$ instead of $m-2$. Indices $m$ range over the set which contains
all even integers and negative odd integers. The sums will therefore
cancel leaving only the term corresponding to $m=-1$ in the second
one,  
\[ f_{\ell}(\tau^{-2}_{\ell}\Gamma_{\ell}(x)) (\tau^{-2}_{\ell}\Gamma_{\ell})'(x) -
f_{\ell}(\Gamma(x))\Gamma'(x) = (\psi_{-1,\ell}\Gamma_{\ell})'(x)
f_{\ell}((\psi_{-1,\ell}\circ\Gamma_{\ell})(x)) \; .\]

But 
\[ \psi_{-1,\ell}\circ\Gamma_{\ell}(x) =\phi_{\ell}^{-1}
\tau_{\ell}^{-1}\tau_{\ell}\phi_{\ell}\tau_{\ell}^{-2}(x) =
\tau^{-2}_{\ell}(x)\]
and the claim of the Lemma follows.
\end{proof}

\paragraph{A functional equation for $\hat{\Gamma}_{\ell}$.} 
By analytic continuation, the formula of Lemma~\ref{lem:26gp,1} is
valid for $f_{\ell,0}$ on a neighborhood of the interval
$(\tau^2_{\ell}x_{\ell},X_{\ell})$. 

Recall $\hat{\Gamma}_{\ell}$ introduced by formula~(\ref{equ:26gp,3})
and similarly define $\hat{T}_{\ell} : = \mu_{\ell} \circ
\tau^{-2}_{\ell}\circ \mu^{-1}_{\ell}$. We obtain a functional
equation in the form
\begin{lem}\label{lem:26gp,2}
For $z$ in a fixed simply-connected neighborhood of $0$ and all
$\ell$ sufficiently large,
\[ \hat{T}_{\ell}(z) = (\hat{T}_{\ell} \circ \hat{\Gamma}_{\ell})(z) - \hat{\Gamma}_{\ell}(z)  \; .\]
\end{lem}
\begin{proof}
Since $\hat{\Gamma}_{\ell}$ fixes $0$, the equality holds at $z=0$. 
We will write $w=\mu_{\ell}^{-1}(z)$ for short. Then, the derivative of the right-hand side is
\[
f_{0,\ell}((\tau_{\ell}^{-2}\Gamma_{\ell})(w))\cdot(\tau^{-2}_{\ell}\Gamma_{\ell})'(w)\cdot
f_{0,\ell}^{-1}(z) -  f_{0,\ell}((\Gamma_{\ell})(w)) \cdot
(\Gamma_{\ell})'(w) \cdot f_{0,\ell}^{-1}(z) = \]
\[ = \tau_{\ell}^{-2}f_{0,\ell}(\tau^{-2}_{\ell} w) f_{0,\ell}^{-1}(z) \]
which is seen to be the derivative of the left-hand side.   
\end{proof}

\subsection{A conditional proof of Theorem~\ref{theo:25ga,1}.}
We will show how to reduce Theorem~\ref{theo:25ga,1} to a certain
simpler and more technical statement. 

Choose a point $z_0$ in the domain of $W_{\infty}$, see
Definition~\ref{defi:5hp,1}, and therefore in $W_{\ell}$ for all
$\ell$ large enough.  Moreover, we can choose it so that the entire
disc $D(0,|z_0|)$ is in $W'$. 

From Fact~\ref{fact:26ga,2} $\hat{G}_{\ell}(z) = a_{\ell}z - c_{\ell}z^3 +
O(|z|^4)$ where $c_{\ell}>0$ for $\ell$ large enough or infinite, while
$a_\ell\geq 1$. It
follows that when $|z|$ and $|\arg z - \pi/2|$ are both suitably small,
then for $\ell$ sufficiently large $|(\hat{G}_{\ell})^{-n}(z)|$
is a monotone decreasing sequence for $n\geq 0$.

Let ${\cal C}_{0,\ell}$ for any $\ell$ denote
the straight line arc from $1$ to $z_0$. It is still in $W_{\ell}$ for
all $\ell$ sufficiently large or infinite by the remark above. 

Suppose that $z_0$ is in the domain of $\hat{G}^{-1}_{\ell}$ for all 
$\ell$ sufficiently large or infinite. Then for any such $\ell$ we can define 
$z_{1,\ell} = \hat{G}^{-1}_{\ell}(z_0)$ still in $W_{\ell}$                   
and pick an arc ${\cal C}_{1,\ell}$
to be the straight line segment joining $z_0$ to $z_{1,\ell}$. 
${\cal C}_{\ell}$ is still the domain of $\log\hat{\phi}_{\ell}$ for all
$\ell$ large enough or infinite. Moreover, they are all in $W'$ and
hence in $W_{\ell}$. 
 
Suppose that ${\cal C}_{1,\ell}$ is in the domain of
$\hat{G}^{-n}_{\ell}$ for all $n$ and $\ell$ large or infinite. Then we
can define points $z_{n,\ell} := \hat{G}_{\ell}^{-n+1}(z_{1,\ell})$ and
arcs ${\cal C}_{n,\ell} := \hat{G}^{-n+1}_{\ell}({\cal C}_1)$ joining
them. They are all in the corresponding $\log\hat{\phi}_{\ell}$ by the functional
equation~(\ref{equ:26ga,3}) and in $W_{\ell}$ since the points tend to
$0$.

Define 
\begin{equation}\label{equ:26gp,2}
s_{n,\ell} = \Re \int_{{\cal C}_{n,\ell}}
\log\frac{\hat{\phi}_{\ell}(z)}{\mu_{\ell}^{-1}(z)}\, dz \; . 
\end{equation}
 
We will rely on the following statement to be proved later. 
\begin{prop}\label{prop:26gp,1}
There is a choice of $z_0$ arbitrarily close to $0$ with $\arg z_0$
less than $\pi/2$ but 
arbitrarily close to $\pi/2$  so that for all
$\ell$ sufficiently large or infinite
\begin{itemize}
\item
arcs ${\cal C}_{n,\ell}$ belong to the domains of the corresponding
$\hat{G}_{\ell}^{-n}$  for all positive $n$,
\item
there is a
constant $Q_4$ independent of $\ell$ such that for all $n>1$
$|s_{n,\ell}| \leq Q_4 \frac{\log n}{n^{3/2}}$. 
\item 
\[ \sup_{t\rightarrow 1} \frac{|\Re z - \Re x_{0,\infty}|}{|z-x_{0,\infty}|^3\log|z-x_{0,\infty}|^{-1}}
< \infty\]
for $z\in\theta_{\infty}$.
\end{itemize}
\end{prop}

Since $0$ 
is a repelling fixed point of
$\hat{G}_{\ell}$ for $\ell$ finite, arcs ${\cal C}_{\ell} :=
\bigcup_{n=0}^{\infty} {\cal C}_{n,\ell}$ are Lipschitz arcs in the
corresponding simply-connected sets $W_{\ell}$.  
By Lemma~\ref{lem:26ga,1} and Corollary~\ref{coro:26ga,1} for every
$\ell$ sufficiently large but finite 

\[ \sum_{n=0}^{\infty} s_{n,\ell} = \log |\tau_{\ell}| \vartheta(\ell)
\; .\]                                                              

From the convergence statement, see Fact~\ref{fact:24gp,1}, for any fixed $n$  
$\lim_{\ell\rightarrow\infty} s_{n,\ell} = s_{n,\infty}$. This
convergence is dominated by a summable series by
Proposition~\ref{prop:26gp,1} and so by the theorem of Lebesgue    
\[ \lim_{\ell\rightarrow\infty} \log|\tau_{\ell}|\vartheta(\ell) = 
\sum_{n=0}^{\infty} s_{n,\infty}\; ,\]   
hence                                                        
\begin{equation}\label{equ:27ga,3}  \lim_{\ell\rightarrow\infty} \vartheta(\ell) =
(\log|\tau_{\infty}|)^{-1} \sum_{n=0}^{\infty} s_{n,\infty} \; .  
\end{equation}

\paragraph{Proof of Theorem~\ref{theo:25ga,1}.}
The first claim follows by the construction. The second one is true
since the simply connected set in question is just $W_{\infty}$.  
The third claim follows directly from Proposition~\ref{prop:26gp,1}.

As for the formula, 
\[ \lim_{\epsilon\rightarrow 0}
\;\Re\left[ \int_0^{1-\epsilon}
  \log\frac{\phi_{\infty}(\theta(t))}{\theta(t)} f_{0,\infty}(\theta(t)) \frac{d\theta}{dt} \,
dt \right] = \lim_{N\rightarrow\infty} \sum_{n=0}^{\infty}
s_{n,\infty} = (\log|\tau_{\infty}|)\lim_{\ell\rightarrow\infty}
\vartheta(\ell) \]
from formula~(\ref{equ:27ga,3}). 

\subsection{Proof of Proposition~\ref{prop:26gp,1}.}
Proposition~\ref{prop:26gp,1} will be derived from
Theorem~\ref{theo:21ga,1} stated in the last section of the
paper. Function $g(z,\eta)$ mentioned there is $-i\hat{G}^{-1}_{\ell}(iz)$
where $\eta:=\eta(\ell)$ has to be related to $\ell$. Let us write    
$\rho(\ell) = (\hat{\Gamma}_{\ell}^{-1})'(0)$. Then $dg(z,\eta)/dz =
\rho(\ell)^2$ by Fact~\ref{fact:26ga,2} 
so that 
\[ \eta(\ell) = \rho(\ell)^{-4}-1\; . \] 
 By
Facts~\ref{fact:24gp,1} and~\ref{fact:26ga,2}, $\rho(\ell)$ is
negative and greater than $-1$ for all $\ell$ and its limit is $-1$ as
$\ell\rightarrow\infty$. This leads to $\eta(\ell)>0$ and
$\lim_{\ell\rightarrow\infty} \eta(\ell)=0$. 

The third coefficient $c(\eta(\ell))$ is real since the coefficients
of $\hat{G}^{-1}_{\ell}$ are all real and $c(0)$ is positive from
Fact~\ref{fact:26ga,2} again. It is also easy to see that the second
and fourth coefficients of $g(z,\eta(\ell))$ are imaginary, but it is
less clear that they are are small with $\eta(\ell)$ as needed
in the setting of Theorem~\ref{theo:21ga,1}.  

\paragraph{Calculation of ${\hat\Gamma}^{-1}$ at $0$.}
Our basic tool will be Lemma~\ref{lem:26gp,2}, which will be used in
the form 
\[ \hat{T}_{\ell}\circ\hat{\Gamma}_{\ell}^{-1}(z) = \hat{T}_{\ell}(z)
- z \; .\]

Taking the second derivative at $z=0$, we derive
\[ (\hat{\Gamma}_{\ell}^{-1})''(0) =
\frac{(\hat{T}_{\ell})''(0)}{(\hat{T}_{\ell})'(0)} (1-\rho(\ell)^2) \;
.\]
The nonlinearity of $\hat{T}_{\ell}$ at $0$ tends to its limit value
for $\hat{T}_{\infty}$ and is bounded for $\ell$ large enough. Then, 
\[  0< 1-\rho(\ell)^2 < 1-\rho(\ell)^4= \frac{\eta(\ell)}{1+\eta(\ell)} < \eta(\ell) \; .\]             
Hence, 
\[ (\hat{\Gamma}_{\ell})^{-1}(z) = \rho(\ell) z + r(\ell)\eta(\ell) z^2 +
s(\ell) z^3 + t(\ell) z^4 + O(|z|^5) \; .\]

Computing the expansion of $\hat{G}_{\ell}^{-1}(z)$ is easy up to the third order terms:                \[ (\hat{G}_{\ell})^{-1}(z) = \rho(\ell)^2 z +
\left[ (\rho(\ell)+1)r(\ell)\rho(\ell)\eta(\ell) \right] z^2 + s'(\ell)z^3 + O(|z|^4) = \]
\[ \rho(\ell)^2 z +
\left[ k_1(\ell)r(\ell)\rho(\ell)\eta^2(\ell) \right] z^2 + s'(\ell)z^3 + O(|z|^4) \]
with the substitution $1+\rho(\ell) = k_1(\ell)\eta(\ell)$ where $k_1(\ell)$ is bounded for all 
$\ell$. 

To determine the fourth coefficient observe that 
\[ (S\hat{G}_{\ell}^{-1})'(0) = (\hat{G}_{\ell}^{-1})^{(4)}(0)/(\hat{G}_{\ell}^{-1})'(0) + 
k_2(\ell)\eta(\ell) \]
with $k_2(\ell)$ similarly bounded for all $\ell$. On the other hand, from the formula
\[ S\hat{G}_{\ell}^{-1} = (S\hat{\Gamma}_{\ell}^{-1})\circ\Gamma_{\ell}^{-1}\cdot 
\left[ (\Gamma_{\ell}^{-1})' \right]^2 + S\hat{\Gamma}_{\ell}^{-1} \] 
one gets                                                                        
\[ (S\hat{G}_{\ell}^{-1})'(0) = (S\hat{\Gamma}_{\ell}^{-1})'(0)(\rho(\ell)^3+1) + 
2(S\hat{\Gamma}_{\ell}^{-1})(0)\cdot (\hat{\Gamma}_{\ell}^{-1})''(0)  \cdot (\hat{\Gamma}_{\ell})^{-1})'(0) = 
k_3(\ell)\eta(\ell) \]
with $k_3(\ell)$ bounded for $\ell$ since $\rho(\ell)^3+1$ and $(\Gamma_{\ell}^{-1})''(0)$ are both 
$O(\eta(\ell))$. Hence, 
\[ (\hat{G}_{\ell}^{-1})^{(4)}(0) = \rho(\ell)^2(k_3(\ell)-k_2(\ell))\eta(\ell) \]
which leads to the form of
the coefficients required by Theorem~\ref{theo:21ga,1}.  

\paragraph{Derivation from Theorem~\ref{theo:21ga,1}.}
Arcs ${\cal C}_{1,\ell}$ are straight line segments which join $z_0$
to $z_{1,\ell}$, which corresponds to $z_{1}(\eta(\ell))$ in
Theorem~\ref{theo:21ga,1}. From the first claim of that Theorem they 
belong to the domain of $\hat{G}_{\ell}^{-n}$ for all $n$. Moreover,
since they can be enclosed in a fixed open set, for all $\eta(\ell)$
sufficiently small the distortion of $\hat{G}^{-n}_{\ell}$ 
on ${\cal
  C}_{1,\ell}$ is bounded independently of $n$ and $\ell$ leading to 
\begin{equation}\label{equ:27ga,4}
|{\cal C}_{n,\ell}| \leq L_1 |z_{n+1,\ell} - z_{n,\ell}| \leq L_1 Q_2
n^{-3/2}
\end{equation}
where the last estimate and constant $Q_2$ come from
Theorem~\ref{theo:21ga,1}. From formula~(\ref{equ:26gp,2}), 
\[ |s_{n,\ell}| \leq |\Re \int_{{\cal C}_{n,\ell}}
\log\hat{\phi}_{\ell}(z)\, dz| + \int_{{\cal C}_{n,\ell}}
|\log\mu^{-1}(z)|\, |dz| \; .\]

$\mu^{-1}(z)$ tends to $\tau^2_{\ell}X_{\ell}$ as $z$ tends to $0$, so
the second integral is bounded 
\[  \int_{{\cal C}_{n,\ell}} |\log\mu^{-1}(z)|\, |dz| \leq L_2 |{\cal
  C}_{n,\ell}|\leq L_2 L_1 Q_2 n^{-3/2} \]
from estimate~(\ref{equ:27ga,4}). 

By the functional equation~(\ref{equ:26ga,3}), for $z\in {\cal
  C}_{n,\ell}$
\[ \log\hat{\phi}_{\ell}(z) = 2(n-1)\log|\tau_{\ell}| +
\log\hat{\phi}_{\ell}(\hat{G}^{n-1}(z)) \; .\]
But $\hat{G^{n-1}}(z)$ belongs to the initial arc ${\cal C}_{1,\ell}$
where $\log\hat{\phi}_{\ell}$ is bounded independently of $n$ and $\ell$. Hence, 
\[ \int_{{\cal C}_{n,\ell}}
|\log\hat{\phi}_{\ell}(\hat{G}^{n-1}(z))|\, |dz| \leq L_3 n^{-3/2} \]
as for the previous term. 

We are left with 
\[ |\Re  \int_{{\cal C}_{n,\ell}} 2(n-1)\log|\tau_{\ell}|\, dz| \leq
L_4 n |\Re \int_{{\cal C}_{n,\ell}}\, dz| = L_4 n |\Re (z_{n+1,\ell} -
z_{n,\ell})| \leq L_4 Q_3 n^{-3/2} \]
from the final claim of Theorem~\ref{theo:21ga,1}, where we have to
remember that the Theorem is stated in a system of coordinates which
swaps the real and imaginary axes and therefore real and imaginary
parts. 

As regards the third claim of Proposition~\ref{prop:26gp,1}, observe
that points $z_{n,\infty}$ are all on the same horizontal line in the
repelling Fatou coordinate and so by the bounded distortion the entire
arc $\theta_{\infty}$ is contained in a horizontal strip of fixed
width in the repelling Fatou coordinate. Since
$-i\hat{G}^{-1}_{\infty}(iz)$ lacks the fourth order coefficient, see
formula~(\ref{equ:21ga,1}), its Fatou coordinate has the form
\[ w = \Psi(z) = Cz^{-2} + O(\log |z|^{-1}) \; .\]
The inverse is
\[ z= \Psi^{-1}(w) = \sqrt{C/w} + O(\log|z|^{-1} |z|^3) \; .\]
By an elementary geometric argument, $\sqrt{C/w}$ will map a
horizontal line to a curve that is third order tangent to the real
axis  at $0$, so that actually the $O(\log|z|^{-1} |z|^3)$ dominates
the imaginary value of the corresponding $z$. This leads to the third
claim of Proposition~\ref{prop:26gp,1} and since the fourth one has
already been checked, the proof of the Proposition is now finished.

\section{Uniform Estimates}
\paragraph{Hypotheses of the mapping.}
We consider a family of maps in the form
\begin{equation}\label{equ:21ga,1} 
g(z,\eta) = \frac{1}{\sqrt{1+\eta}}z + b(\eta)\eta^2 z^2 - c(\eta)z^3
+ d(\eta)\eta z^4 + z^5 R(z,\eta) \; .
\end{equation}

Here $z$ is a complex variable and $\eta$ a positive parameter from
some set which has $0$ as an accumulation point. All $g(z,\eta)$ are
assumed to be holomorphic for $z$ in some neighborhood of $0$, which
is independent of $\eta$. 
Coefficients $b(\eta), c(\eta), d(\eta)$ are complex and assumed to have limits $
b(0),c(0),d(0)$, respectively. Functions $R(z,\eta)$ also converge uniformly
to $R(z,0)$ as $\eta\rightarrow 0$.   Additionally, assume that for
all $\eta$ coefficients $b(\eta)$ are imaginary, $c(\eta)$ are real
and the limit $c(0)$ is strictly positive. 

From the dynamical point view we observe that for all $\eta$ function
$g(z,\eta)$ which has an attracting fixed point at $0$. As
$\eta\rightarrow 0$ functions $g(z,\eta)$ converge uniformly to
$g(z,0)$ which as a neutral point which is of the third order and still topologically
attracting along the positive real axis. The second and fourth
coefficients are assumed to vanish for the limiting map and at definite
rates as $\eta\rightarrow 0$. 

If $z_0$ is a given point, we will write
$z_n(\eta):=g^n(z_0,\eta)$ where the composition here and
elsewhere is on the first coordinate. 
\begin{theo}\label{theo:21ga,1}
For a family functions $g(z,\eta)$ given by formula~(\ref{equ:21ga,1})
with the hypotheses underneath, there exists positive parameters
$\alpha_0,\eta_0,\delta_0$, together with positive constants
$Q_1,Q_2,Q_3$  and and open set $U_Q$ such
that the following claims hold whenever $|z_0|<\delta_0$, $|\arg
z_0|<\alpha_0$ for all $0<\eta<\eta_0$ and all integers $n>1$:
\begin{itemize}
\item
$U_Q$ contains $z_0,z_1(\ell)$ and the straight
line segment which connects them, and for every $z\in U_Q$ iterates 
$g^n(z,\eta)$ are defined,
\item 
\[ |z_n(\eta)| \leq Q_1\sqrt{\max(n^{-1},\eta)}\exp(-n\eta/8) \; ,\]
\item
\[ |z_{n+1}(\eta)-z_n(\eta)| \leq Q_{2} n^{-3/2}\; ,\]
\item
\[ |\Im z_{n+1}(\eta)- \Im z_n(\eta)| \leq Q_{3} n^{-5/2}\log n\; .\]
\end{itemize}
\end{theo}

\subsection{Reduction to the case of $b(\eta)=0$.}
By a change of coordinates we can eliminate the ``small'' second order 
term completely. 

Consider the change of coordinates $z(u,\eta) =u - a(\eta) u^2$, where                                \[ a(\eta)  =  b(\eta)\eta(1+\eta)(\sqrt{1+\eta}+1) \]                        
which is $O(\eta)$ for $\eta$ small. By a direct calculation one
verifies that this change of coordinates removes the second order
coefficient from $g(z,\eta)$. At the same time, since every
coefficient of $z(u,\eta)$ and the inverse function, except for the
first-order ones is $0(\eta)$, $c(\eta)$, $d(\eta)\eta$ and $R(z,\eta)$
change only by $O(\eta)$. In particular their limits as
$\eta\rightarrow 0$ do not change and one can still factor out $\eta$
from the fourth order coefficients. 

To see that the third coefficient remains real, observe that $a(\eta)$
is imaginary and hence the Schwarzian derivative of $z(u,\eta)$ at
$u=0$ is real. By the chain rule for Schwarzian derivatives the same is
true for the inverse map. As regards $g(z,\eta)$, its Schwarzian
derivative at $0$ is also real under our assumptions. Then by the chain
rule again the Schwarzian derivative at $0$ for the conjugated map is
real, and since its first coefficient is real and the second one is
missing, that means that the third one is real.  

So, the change of coordinates leaves
$g(z,\eta)$ in the form postulated at the beginning 
and additionally sets $b(\eta)$ to $0$. 

Looking now at Theorem~\ref{theo:21ga,1} we observe that the
hypotheses on $z_0$ preserve their meaning in the new coordinate for
the point $u_0=u(z_0,\eta)$. 
Its claims are now for points $u_n(\eta)$ and have to verified for 
$z_n(\eta) = z(u_n(\eta),\eta)$. The first claim is satisfied for
$\eta$ sufficiently small, since the distortion of the change of
coordinates tends to $0$ with $\eta$.
Since $\lim_{u\rightarrow 0}
\frac{z(u,\eta)}{u}=1$ uniformly for all bounded $\eta$, 
the second claim is satisfied. 

For the third one, we estimate
\[ |z(u_{n+1}(\eta),\eta) - z(u_n(\eta),\eta)| \leq
|u_{n+1}(\eta)-u_n(\eta)| \left[ 1 + O(\eta)|u_{n+1}(\eta)+u_{n}(\eta)
  \right] \; .\]

The expression in the square parentheses is bounded by the second
claim, so the third claim follows.  

\paragraph{Verification of the fourth claim.}
We state a simple fact which will be useful in several estimates.
\begin{lem}\label{lem:23ga,1}
For $\alpha,\beta>0$ the function $x^{\beta}e^{-\alpha x}$ considered
on the set $x>0$ has the maximum equal to $(\beta/\alpha)^{\beta}e^{-\beta}$.
\end{lem}
\begin{proof}
Elementary exercise.
\end{proof} 

Regarding the fourth claim, estimate
\[ |\Im z(u_{n+1}(\eta),\eta) - \Im z(u_n(\eta),\eta) | \leq \] 
\[ \leq |\Im
u_{n+1}(\eta) - \Im u_{n}(\eta)| + L_1 \eta (|u_{n+1}(\eta)| +
  |u_n(\eta)|) |u_{n+1}(\eta)-u_n(\eta)| \leq \]
\[ \leq |\Im
u_{n+1}(\eta) - \Im u_{n}(\eta)| +  
2L_2 \eta \exp(-n\eta/8)\sqrt{\max(n^{-1},\eta)} 
|u_{n+1}(\eta)-u_n(\eta)|  \]
from the second claim, where $L_1,L_2$ are constants independent of
$\eta$.  Let us concentrate on the expression which follows $L_2$ and take
first $n\leq \eta^{-1}$. Then 
\[  \eta \exp(-n\eta/8)\sqrt{\max(n^{-1},\eta)} \leq n^{-3/2} \; .\]
For $n>\eta^{-1}$, use Lemma~\ref{lem:23ga,1} with $\beta=1$ to get 
\[ \exp(-n\eta/8) \leq (8/e)\eta^{-1} n^{-1} \]
which then leads to 
\[ \eta \exp(-n\eta/8)\sqrt{\max(n^{-1},\eta)} \leq (8/e) n^{-1} \sqrt{\eta}
< (8/e) n^{-3/2} \; .\]

Now using the third and fourth claims of the theorem for $u$ we get 
\[ |\Im z(u_{n+1}(\eta),\eta) - \Im z(u_n(\eta),\eta) | \leq Q_{4,u}
n^{-5/2}\log n + L_3 Q_{3,u} n^{-3} \; .\]

\subsection{Pre-Fatou coordinate and the proof of the first claim.}
Inspired by the Fatou coordinate for the limiting map $g(z,0)$, we
change the variable to $w(z) = z^{-2}$ on the set $\{ z :\: \Re z > 0\}$.
The map corresponding to $g$ is then 
\begin{equation}\label{equ:21ga,2} \gamma(w,\eta) = (1+\eta) w  +
\hat{c}(\eta) + \frac{\hat{d}(\eta)\eta}{\sqrt{w}} +
\frac{\hat{R}(w,\eta)}{w}
\end{equation}

where                                                       
\[ \begin{array}{ccc}
\hat{c}(\eta) &=& 2c(\eta)(\sqrt{1+\eta})^3\\
\hat{d}(\eta) &=& -2d(\eta)(\sqrt{1+\eta})^3
\end{array} \]
and $\hat{R}(w,\eta)$ are holomorphic functions of $\sqrt{w}$. 
Coefficients $\hat{c}(\eta),\hat{d}(\eta)$ and functions
$\hat{R}(w,\eta)$ still have limits as $\eta\rightarrow 0$ and
$\hat{c}(0)=2c(0)>0$.     

\begin{lem}\label{lem:23ga,2}
There is $W_0\geq 1$ such that whenever $\Re w \geq W_0$, then for all
$\eta$ sufficiently small 
\[ \Re \gamma(w,\eta) > \Re w + c(0) \; .\]
\end{lem}
\begin{proof}
From the form of~(\ref{equ:21ga,2}) for $\eta$ small,
the non-linear part of $\gamma(w,\eta)$ tends to $0$ as $|w|\rightarrow\infty$.
Also, for $\eta$ small enough $c(0) \leq \frac{2}{3}\hat{c}(\eta)$. Then 
\[ \Re \gamma(w,\eta) \geq \Re w (1+\eta) + c(0) \]
and the claim of the Lemma follows immediately. 
\end{proof}

We will write $w_n(\eta) = w(z_n(\eta))$.
By setting $\delta_0$ small enough in the hypothesis of
Theorem~\ref{theo:21ga,1}, we can ensure $\Re w_o(\eta) \geq W_0 + 1$ and then 
$\Re w_n(\eta) \geq nc(0)$. The region $\{ w :\: \Re w > W_0\}$ is
invariant under $\gamma(w,\eta)$ and $z_n(\eta) =
\frac{1}{\sqrt{w_n(\eta)}}$. 

Then  the first claim of
  Theorem~\ref{theo:21ga,1} follows together with an estimate
\begin{equation}\label{equ:23ga,1} 
|z_{n}|(\eta) \leq (\Re w_n(\eta))^{-1/2} \leq (c(0))^{-1/2} n^{-1/2}
\end{equation}
for all $n$. 
The open set can be taken to be the image on $\{ w :\: \Re w >
W_0\}$ by $z=w^{-1/2}$. The distortion of this map can be made
arbitrarily small by taking $W_0$ large and then it will contain the
straight-line segment from $z_0$ to $z_1(\eta)$.

\paragraph{Estimates of derivatives of $\gamma$.}
From formula~(\ref{equ:21ga,2}), we conclude that with some $K>0$
\begin{equation}\label{equ:21gp,1}
 | \frac{d\gamma}{dw}(w,\eta) (1+\eta)^{-1} - 1 | \leq K |w|^{-3/2} 
\end{equation}

for all $\eta$ sufficiently small and all $\Re w \geq W_0$. 

\begin{lem}\label{lem:21gp,1}  
The exists $W_1 \geq W_0$ and $K_1>0$ such that for all $\eta$ small enough, all
$w$ which satisfy $\Re w \geq W_1$ and all $n>0$
 \[ | \log\frac{d\gamma^n(w,\eta)}{dw}(w,\eta)(1+\eta)^{-n} | \leq
    \frac{K_1}{\sqrt{\Re w}} \; .\]
\end{lem}
\begin{proof}
Start with estimate~(\ref{equ:21gp,1}) and choose $W_1$ so that $\Re
w\geq W_1$ implies 
  \[ |\log\frac{d\gamma(w,\eta)}{dw}(w,\eta)(1+\eta)^{-1}| \leq 2K|w|^{-3/2}
    \; .\]
Since $W_1 \geq W_0$ we have $\Re g^n(w,\eta) \geq \Re w +nc(0)$. From the
chain rule 
\[   | \log\frac{d\gamma^n(w,\eta)}{dw}(w,\eta)(1+\eta)^{-n} | \leq
    \sum_{k=0}^{n-1} \frac{2K}{(\Re w +c(0)k)^{3/2}} \leq
    \frac{K_1}{\sqrt{\Re w}} \]
by comparing the sum with an integral.
\end{proof}

\begin{coro}\label{coro:21gp,1}
There is a constant $K_2$ such that for all $\eta$ small enough and if
$\Re w\geq W_0$, one gets the estimate
 \[ | \log\frac{d\gamma^n(w,\eta)}{dw}(w,\eta)(1+\eta)^{-n} | \leq K_2
 \; .\]
\end{coro}
This follows from Lemma~\ref{lem:21gp,1}, since the derivative
$\frac{dg}{dw}$ is bounded uniformly for all $\Re w \geq W_0$ and
$\eta$ small enough and since $\Re \gamma^k(w,\eta) \geq \Re w + c(0)k$
so that after $\frac{W_1-W_0}{c(0)}$ steps the hypothesis of
Lemma~\ref{lem:21gp,1} is satisfied. 

\subsection{Linear approximation and the second claim.}
The main tool for our handling of the estimates for function
$\gamma(w,\eta)$ will be its approximation by the linear part
$\gamma_0(w,\eta) = (1+\eta)w + c(0)$. 

\begin{lem}\label{lem:21gp,2}
There are positive constants $K_3,K_4$ such that for all $\Re w \geq
W_0$ and $\eta$ sufficiently small and for all positive integers $n$
the estimate
\[ (1+\eta)^{-n} \left|\gamma^n(w,\eta)-\gamma_0^n(w,\eta)\right| \leq
K_3 + K_4\min(\log n, -\log\eta) \]
holds.
\end{lem}
\begin{proof}
Start with an elementary identity:
\[ \gamma^n(w,\eta) = \] 
\[ = \gamma_0^n(w,\eta) + \sum_{k=1}^n \left[
  \gamma^{n-k}(\gamma(\gamma_0^{k-1}(w,\eta),\eta),\eta) -  
\gamma^{n-k}(\gamma_0(\gamma_0^{k-1}(w,\eta),\eta),\eta) \right]\; .\]

This yields an estimate 
\[ (1+\eta)^{-n} \left|\gamma^n(w,\eta)-\gamma_0^n(w,\eta)\right| \leq \]
\[\leq\sum_{k=1}^n  (1+\eta)^{-n}|\frac{d\gamma^{n-k}(\tilde{w}_k,\eta)}{dw}|
|\gamma(\gamma_0^{k-1}(w,\eta),\eta) -
\gamma_0(\gamma_0^{k-1}(w,\eta),\eta)| \]
where $\tilde{w}_k$ is a point on the segment joining points 
\[ \gamma(\gamma_0^{k-1}(w,\eta),\eta)\] and
\[ \gamma_0(\gamma_0^{k-1}(w,\eta),\eta)\; .\]
Those derivatives can be
estimated by Corollary~\ref{coro:21gp,1} which results in 
\[ (1+\eta)^{-n} \left|\gamma^n(w,\eta)-\gamma_0^n(w,\eta)\right| \leq \]
\[ \leq K_2  \sum_{k=1}^n (1+\eta)^{-k}
|\gamma(\gamma_0^{k-1}(w,\eta),\eta) -
\gamma_0(\gamma_0^{k-1}(w,\eta),\eta)|\;  \]

From~(\ref{equ:21ga,2}) we obtain
\[ |\gamma(\gamma_0^{k-1}(w,\eta),\eta) -
\gamma_0(\gamma_0^{k-1}(w,\eta),\eta)| \leq
\frac{2|\hat{d}(0)|\eta}{\sqrt{|\gamma_0^{k-1}(w,\eta)|}} +
\frac{Q}{|\gamma_0^{k-1}(w,\eta)|} \; .\]
Since $|\gamma_0^{k-1}(w,\eta)| \geq W_0+(k-1)c(0)$, we get 
\[ (1+\eta)^{-n} \left|\gamma^n(w,\eta)-\gamma_0^n(w,\eta)\right| \leq\]
\[ \leq K_2\sum_{k=0}^{n-1} (1+\eta)^{-k-1}\left[
  \frac{2|\hat{d}(0)|\eta}{\sqrt{W_0+c(0)k}} +
  \frac{Q}{W_0+c(0)k}\right] \; .\]
 
Applying Abel's transformation yields
\[  \sum_{k=0}^{n-1} (1+\eta)^{-k-1}
  \frac{2|\hat{d}(0)|\eta}{\sqrt{W_0+c(0)k}} \leq L_1 +
  L_2\frac{1}{\eta}\sum_{k=0}^{n-1} (1+\eta)^{-k-1}\frac{\eta}{
(W_0+c(0)k)^{3/2}} \]
with constants $L_1$, $L_2$ independent of $\eta$. The factor $\eta$
in the numerator came from the form of the fourth coefficient and it
cancels with $\eta$ in the denominator which arises from the summation
of $(1+\eta)^{-k}$. This is clearly
uniformly bounded for all $n$ and can be included in $K_3$ in the
claim of Lemma~\ref{lem:21gp,2}.

\[ \sum_{k=0}^{n-1} (1+\eta)^{-k-1}  \frac{Q}{W_0+c(0)k} \leq L_1 +
L_2\log n \;\]
trivially from a comparison with the integral. To get alternative
estimate claimed in the Lemma, introduce an integer-valued function
\begin{equation}\label{equ:22ga,1}
N(\eta) = \left\lfloor 1/\eta \right\rfloor \;. 
\end{equation}
Then for $n\geq N(\eta)$ we split
\[  \sum_{k=0}^{n-1} (1+\eta)^{-k-1}  \frac{Q}{W_0+c(0)k} =  \]
\[ = \sum_{k=0}^{N(\eta)-1} (1+\eta)^{-k-1}  \frac{Q}{W_0+c(0)k} +
\sum_{k=N(\eta)}^{n-1} (1+\eta)^{-k-1}  \frac{Q}{W_0+c(0)k} \leq \]
\[ \leq L_1 +
L_2\log N(\eta) + L_3 + \frac{L_4}{\eta}\sum_{k=N(\eta)}^{n-1}
  (1+\eta)^{-k-1}(W_0+kc(0))^{-2} \]
where we used the trivial estimate for the first sum and Abel's
transformation for the second one. The term which follows $L_4$ is
estimated by 
\[ \eta^{-1} \sum_{k=N(\eta)}^{n-1}
  (1+\eta)^{-k-1}(W_0+kc(0))^{-2} \leq \eta^{-1} \sum_{k=N(\eta)}^{n-1}
  (W_0+kc(0))^{-2} \leq \] 
\[ \leq \eta^{-1} \int_{N(\eta)-1}^{\infty}
\frac{dx}{(W_0+xc(0))^{2}} \leq 
  \frac{1}{\eta c(0) (W_0+c(0)(N(\eta)-1))} \leq \frac{L_5}{\eta (N(\eta)+1)} \]
for $\eta$ small enough.

From the definition of $N(\eta)$  we get $\eta (N(\eta)+1) > 1$ and so
this term is uniformly bounded. Also, $\log N(\eta) \leq -\log\eta$. 
Thus, finally,
\[  \sum_{k=0}^{n-1} (1+\eta)^{-k-1}  \frac{Q}{W_0+c(0)k} \leq K_4
\log\eta^{-1} \; .\]
\end{proof}

\paragraph{Proof of the second claim in Theorem~\ref{theo:21ga,1}.}
As an elementary exercise we obtain
\begin{equation}\label{equ:22gp,1}
\gamma_0^n(w,\eta) = (1+\eta)^n w + c(0)\frac{(1+\eta)^n-1}{\eta} \; .
\end{equation}
From Lemma~\ref{lem:21gp,2}
\begin{equation}\label{equ:22gp,3}
 \Re w_n(\eta) = \Re\gamma^n(w_0,\eta) \geq (1+\eta)^n\Re w_0 +
(1+\eta)^n(c(0)\eta^{-1} - K_3 + K_4\log\eta) - c(0)\eta^{-1} \; .
\end{equation}
The second claim of Theorem~\ref{theo:21ga,1} is easy for $n\leq
N(\eta)$, recall formula~(\ref{equ:22ga,1}). Then the exponential term
is bounded below by $e^{-1}$ and the second claim says only the same
as estimate~(\ref{equ:23ga,1}). 

So let $n>N(\eta)$. For $n\geq N(\eta)$ and $\eta$ sufficiently small, we have 
\begin{itemize}
\item
$ c(0)\eta^{-1} - K_3 +K_4\log\eta > c(0)\eta^{-1} \exp(-0.1) $,
\item
$(1+\eta)^n \geq \exp(n\eta-0.1)$.
\end{itemize}

This leads  to 
\[ \Re w_n(\eta) \geq \Re w_0 + c(0)\left[ \exp(n\eta-0.2)-1 \right] \eta^{-1} \; .\]
Since $n>N(\eta)$ we have $n\eta>1$ and note an elementary inequality 
$e^x - 1 > e^{x/4}$ for $x\geq 0.8$. Then, 
\begin{equation}\label{equ:22gp,2}
\Re w_n(\eta) \geq \Re w_0 + c(0)\exp(n\eta/4)\eta^{-1} \; .
\end{equation}

Then 
\[ |z_n(\eta)| = |w_n(\eta)|^{-1/2} \leq
\exp(-n\eta/8)\sqrt{\eta/c(0)} \; .\]
and the second claim of Theorem~\ref{theo:21ga,1} is satisfied. 

\subsection{Proofs of the remaining claims.}
We start with a lemma which extracts some estimates from the
discussion of the preceding section. 
\begin{lem}\label{lem:22ga,1}
There are constants $K_5, K_6$ such that for all $\eta$ sufficiently small
and $n>1$, the following estimates hold:
\[ \begin{array}{ccc} \Re w_n(\eta) &\geq& \Re w_0 + K_5
  \min(n,\eta^{-1})\exp(n\eta/4)\\
|\Im w_n(\eta)| &\leq & K_6 |\Re w_n(\eta)|  \max(\frac{\log n}{n},-\eta\log\eta)\end{array}   \; .\]
\end{lem}
\begin{proof}
The first estimate is  just the estimate~(\ref{equ:22gp,2}) in the
case when $n>N(\eta)$ and minimum comes from using the 
estimate of 
Lemma~\ref{lem:23ga,2} for smaller $n$. 

To prove the second one, we need to return to Lemma~\ref{lem:21gp,2}
and formula~(\ref{equ:22gp,1}). Observe that $\gamma_0$ preserves the
imaginary part and so 
\[ |\Im w_n(\eta)| \leq (1+\eta)^n \left[ K_3 + |\Im w_0| +
  K_4\min(\log n,-\log \eta) \right]\;
.\]
Since $w_0$ is the fixed initial point, it does not depend on
$\eta$. So we can rewrite the estimate as for $n>1$ and $\eta$
sufficiently small. 
\[ |\Im w_n(\eta)| \leq L_1 (1+\eta)^n \min(\log n,-\log \eta)\; .\]
 
For $n\leq N(\eta)$  the growth of $\Re w_n(\eta)$ is at least linear
by Lemma~\ref{lem:23ga,1} while the exponential term in the preceding
estimate is bounded by $e$ and so the second estimate will hold with
$\frac{\log n}{n}$ picked in the maximum. 

For $n > N(\eta)$, we have to come back to
estimate~(\ref{equ:22gp,3}), which yields
\[ \Re w_n(\eta) \geq (1+\eta)^n c(0)\eta^{-1}(0.9 - (1+\eta)^{-n})
\geq (1+\eta)^n c(0)\eta^{-1}/2 \]
for $\eta$ sufficiently small. Now the second estimate also follows
with the second possibility in the maximum. 
\end{proof}

\begin{lem}\label{lem:22gp,1}
There is a constant $K_7$ such that for every $\eta$ sufficiently
small and $n>1$
\[ |\Im z_n(\eta)| \leq K_7 \max(n^{-3/2}\log n, -\eta^{3/2}\log\eta) \exp(-n\eta/8)   \; .\]
\end{lem}
\begin{proof}
Clearly, 
\[ |\Im z_n(\eta)| = \left| \Im(\frac{1}{\sqrt{\Re w_n(\eta) + i \Im w_n(\eta)}})\right| =
(\Re w_n)^{-1/2} |\Im \left( 1+i\frac{\Im w_n(\eta)}{\Re w_n(\eta)}
  \right)^{-1/2} |\; .\]
From Lemma~\ref{lem:22ga,1} the ratio of imaginary to the real part
is small for small $\eta$ and large $n$, and so for all $n\leq L_1$,
where $L_1$ does not depend on $\eta$, because the estimate on the
ratio is uniform, and $\eta$ sufficiently small
\[  |\Im \left( 1+i\frac{\Im w_n(\eta)}{\Re w_n(\eta)}
  \right)^{-1/2}| \leq \left| \frac{\Im w_n(\eta)}{\Re w_n(\eta)} \right|\; .\]
So, 
\[ |\Im z_n(\eta)| \leq \left| (\Re w_n)^{-1/2} \frac{\Im w_n(\eta)}{\Re
  w_n(\eta)}\right| \; .\]
From Lemma~\ref{lem:22ga,1}, 
\[  |\Im z_n(\eta)| \leq K_6/\sqrt{K_5}\max(\frac{\log
  n}{n},-\eta\log\eta)\sqrt{\max(n^{-1},\eta)} \exp(-n\eta/8) = \]
\[ =K_7 \max(n^{-3/2}\log n, -\eta^{3/2}\log\eta) \exp(-n\eta/8) \]
since both maxima pick up their for or second value simultaneously
depending on whether $n\leq N(\eta)$ or not.  
\end{proof}

\paragraph{Proof of the third claim.}
From the form of $\gamma(z,\eta)$ we conclude that 
\[ |z_{n+1}(\eta)-z_n(\eta)| \leq L_1
|z_n(\eta)|\max(\eta,|z_n(\eta)|^2) \leq L_2 |z_n(\eta)|\max(\eta,n^{-1}) \]
using estimate~(\ref{equ:23ga,1}). $L$ as usual
denote constant independent of $n$ and $\eta$. 

We see that for $n\leq N(\eta)$, we get 
\[ |z_{n+1}(\eta)-z_n(\eta)| \leq L_3 n^{-3/2} \; .\]

For $n>N(\eta)$ the maximum will be $\eta$ and from the second claim
of the Theorem, we get 
\[  |z_{n+1}(\eta)-z_n(\eta)| \leq L_4 \eta^{3/2} \exp(-n\eta/8) \; .\]
By Lemma~\ref{lem:23ga,1}, with $\beta=3/2$, for any positive $n$, 
\[ \exp(-n\eta/8) \leq n^{-3/2} \eta^{-3/2} (12)^{3/2} \]
and in conclusion 
\[ |z_{n+1}(\eta)-z_n(\eta)| \leq L_5 n^{-3/2} \; .\]

\paragraph{Proof of the fourth claim.}
Taking into account that $\eta$ and $\hat{c}(\eta)$ are real, we
obtain an estimate
\[ |\Im z_{n+1}(\eta) - \Im z_n(\eta) | \leq \] 
\[ \leq \eta |\Im z_n(\eta)| +
3\hat{c}(\eta) |\Im z_n(\eta)| |z_n(\eta)|^2 + |\hat{d}(\eta)|\eta
|z_n(\eta)|^4 + L_1 |z_n(\eta)|^5 \; .\]

Taking into account the estimates of the second claim of
Theorem~\ref{theo:21ga,1} and of Lemma~\ref{lem:22gp,1} gives    
\[ L_2^{-1} |\Im z_{n+1}(\eta) - \Im z_n(\eta) | \leq \] 
\[ \leq \max(n^{-3/2}\log
n, -\eta^{3/2}\log\eta) \exp(-n\eta/8) \left[
  \eta+\max(n^{-1},\eta)\exp(-n\eta/4) \right] + \]
\[ + \max(n^{-2},\eta^2)\exp(-n\eta/2)\left[ \eta +
  \sqrt{\max(n^{-1},\eta)}\right]\; .\]

For $n\leq N(\eta)$, we can continue by 
\[  L_2^{-1} |\Im z_{n+1}(\eta) - \Im z_n(\eta)| \leq L_3 n^{-5/2}\log
n \; .\]

For $n>N(\eta)$ we get
\[  L_2^{-1} |\Im z_{n+1}(\eta) - \Im z_n(\eta)| \leq \]
\[ \leq L_4
\eta^{5/2}\log\eta^{-1}\exp(-n\eta/8) + \eta^{5/2}\exp(-n\eta/2) \leq
2L_4  \eta^{5/2}\log\eta^{-1}\exp(-n\eta/8) \]
for $\eta$ sufficiently small. Lemma~\ref{lem:23ga,1} leads to 
$\exp(-n\eta/8)\leq (20/e)^{5/2} \eta^{-5/2}n^{-5/2}$ and hence to  
\[   L_2^{-1}|\Im z_{n+1}(\eta) - \Im z_n(\eta)| \leq L_5 \log\eta^{-1}
n^{-5/2} < L_5 n^{-5/2}\log n \]
which gives the fourth claim. 

\section{Unimodal Fibonacci maps}\label{sfuni}
We prove that the limiting drift is 
infinite in this case following the scheme which is used for
the Fibonacci covering maps. 
This result is not new and could be extracted from~\cite{bkns}.

We will use the same notations for analogous notions
for the Fibonacci covering and unimodal maps.
By $\ell$ we will denote any {\em even} integer greater than $2$.
The following facts can be obtained from~\cite{bkns},~\cite{ns},~\cite{E},~\cite{buff}.
\begin{fact}\label{uni}
For every $\ell$ there are parameters
$\tau_{\ell}>1$, $x_{0,\ell}\in (0, 1)$ and a map  $\phi_\ell$ with a representation 
\[ \phi_{\ell}(x) = \left( E_\ell(x) \right)^{\ell} \]
as follows:
\begin{itemize}
\item
$E_\ell:I \rightarrow
(-\sqrt[\ell]{\tau^2_{\ell}}, \sqrt[\ell]{\tau^4_{\ell}})$ is a decreasing diffeomorphism from the
Epstein class defined on a real interval $I\supset [0,1]$,   
\item $E_\ell(x_{0,\ell})=0$,   $E_\ell(0) = 1$, 
\item the {\em fixed point equation} holds for all  $x$ where both sides are defined:
\begin{equation}\label{fequni} 
\ \tau_\ell^{-1} \phi_\ell \tau_\ell(x) = (\tau_\ell \phi_{\ell} \tau_\ell^{-1})\circ \phi_\ell(x)\ .
\end{equation} 
\item $T_1\le \tau_{\ell}\le T_2$,
$c_1\le x_{0,\ell}<\tau_\ell x_{0,\ell}<c_2$ for all $\ell$ and some $1<T_1<T_2<\infty$, $0<c_1<c_2<1$.
\end{itemize}
\end{fact}

\paragraph{Induced dynamics.}
For each $\ell$ there is a unique point $X_\ell\in (0, x_{0,\ell})$ 
so that $\phi_\ell(X_\ell)=\tau_{\ell} X_{\ell}$.
It follows from the fixed point equation~(\ref{fequni}) that 
$\phi_{\ell}(\tau_{\ell} X_{\ell})=\tau_{\ell}^3 X_{\ell}$. Note also that 
$\phi_\ell(X_\ell)=\tau_{\ell} X_{\ell}<\tau_\ell x_{0,\ell}=\phi_\ell(x_{0,\ell}/\tau_\ell)$
where the latter equality follows from the equation~(\ref{fequni}).
Hence, $x_{0,\ell}<\tau_{\ell} X_{\ell}$. 
Using Fact~\ref{uni},
$0<c_1/T_2<X_\ell<x_{0,\ell}<\tau_\ell X_\ell<\tau_\ell x_{0,\ell}<c_2<1$ for all $\ell$.

For each $\phi_\ell$ consider its {\em inducing domain} $J_\ell=(X_{\ell}, \tau_{\ell} X_{\ell})$.
\begin{defi}\label{induni}
The {\em induced dynamics} 
\[ \Phi_{\ell} :\:  J_\ell \rightarrow J_\ell\]  is defined with the exception of
countably many points and in the form
\[ \Phi_{\ell}(x) = \tau^{-m(x)}_{\ell} \phi_{\ell}(x) \]
where $m(x)=m(x,\ell)$ is the unique integer for which 
$\tau^{-m(x)}_{\ell} \phi_{\ell}(x) \in J_\ell$. 
\end{defi}
Function $m(x)$ takes values $0,-1,-2,...$ on $(X_\ell, x_{0,\ell})$ and 
$2,1,0,-1,-2,...$ on $(x_{0,\ell}, \tau_{\ell} X_{\ell})$. A crucial difference with 
the drift function for the Fibonacci covering maps is that $-m(x)$ is bounded from below. 
The inverse branches of $\Phi_{\ell}$ can be written as
\begin{equation}\label{bruni}
 \psi_{m,\ell}^\pm(x) = (\phi^{-1}_{\ell})_\pm(\tau^m_{\ell} x) \; .
\end{equation}
Here $(\phi^{-1}_{\ell})_\pm(x)=E_\ell^{-1}(\pm x^{1/\ell})$ and we choose $-$ if 
$\phi^{-1}_{\ell}(x)\in (x_{0,\ell}, \tau_{\ell} X_{\ell})$ and $+$ if 
$\phi^{-1}_{\ell}(x)\in (X_\ell, x_{0,\ell})$ so that
$(\phi^{-1}_{\ell})_{+}'<0$,  $(\phi^{-1}_{\ell})_{-}'>0$. 

It follows from Fact~\ref{uni} that $(\phi^{-1}_{\ell})_{+}$ extends to
a univalent map in  $\CC\setminus\{x\in\RR :\: x\notin [0, 1]\}$ and $(\phi^{-1}_{\ell})_{-}$
extends to a univalent map in $\CC\setminus\{x\in\RR :\: x\notin [0, \tau_\ell^2]\}$
Therefore,
\begin{coro}\label{corouni}
For all $\ell$ sufficiently large the induced map $\Phi_{\ell}$ is
uniformly expanding in the Poincar\'{e} metric of $\CC\setminus\{
x\in\RR :\: x\notin [0, 1]\}$.
\end{coro}

Let $I_0$, $I_1$ be fixed closed segments such
that 
\begin{equation}\label{intuni}
I_0 \subset  \mbox{int}\:(I_1) \subset I_1 \subset (0,1) \; .
\end{equation}
The real bounds from Fact~\ref{uni} imply that if $I_0$ is close enough to $(0,1)$ then
for every $m$ and $\ell$ sufficiently large $\psi_{m,\ell}^\pm(I_1)\subset I_0$. 

\paragraph{Transfer operators and invariant densities.}
The transfer operator is given by the formula
\begin{equation}\label{opuni}
P_{\ell} f(z) = \sum_{m\le 2} (\psi_{m,\ell}^{-})'(z) f\circ \psi_{m,\ell}^{-}(z)-
\sum_{m\le 0} (\psi_{m,\ell}^{+})'(z) f\circ \psi_{m,\ell}^{+}(z) \; .
\end{equation}
\begin{fact}\label{densuni}
For every $\ell$ sufficiently large, the transfer
operator acting on the space $L_1(J_\ell)$
has unique non-negative invariant density $f_{\ell} =
\lim_{n\rightarrow\infty} P_{\ell}^n {\bf e}$, which in fact is 
positive, where ${\bf e}$ is a constant density.  
\end{fact}
This follows from the Folklore Theorem in the light of
Corollary~\ref{corouni}.  

Let $D_0$, $D_1$ be disks with diameters $I_0$, $I_1$ respectively.
The functions $\psi_{m,\ell}^{\pm}$ 
admit holomorphic extensions to the disk whose diameter is $(0,1)$.
Based on Fact~\ref{uni} and~(\ref{intuni}),
$\psi_{m,\ell}(D_1)\subset D_0$ for all $m$ and $\ell$ sufficiently large.
The functional space ${\cal F}_0$ will consist
of all functions bounded and holomorphic on $D_0$, real on the real
line, as well continuous
on $\overline{D_0}$, with the $C^0$ norm. 
Repeating considerations from the proofs of Lemma~\ref{lem:24gp,3}, 
Lemma~\ref{lem:24gp,4} and Lemma~\ref{lem:25ga,1}
we conclude that the operators $P_\ell$ are well defined and compact in the space ${\cal F}_0$,
the densities $f_{\ell}$ admit holomorphic continuations to $f_{0,\ell}\in {\cal F}_0$
and the sequence $(f_{0,\ell})$ is uniformly bounded in this space.

\paragraph{The Drift}
For every $\ell$ we define the {\em drift} as before:
\[ \vartheta(\ell) :=  -\int_{J_\ell} m(x,\ell) f_{\ell}(x)\;
dx \; .\]

We need an {\em associated map} 
\begin{equation}\label{assuni}
\begin{array}{ccc}
G_{\ell} & = & \tau_{\ell}^{-1} \phi_{\ell} \tau_{\ell}^{-1} \\
\end{array}
\end{equation}
It follows from the fixed point equation (see also~\cite{E},~\cite{buff}).
\begin{fact}\label{assunipro}
For any finite $\ell$,
\[ G_{\ell}(x_{0,\ell})=x_{0,\ell}, \ \ G_{\ell}'(x_{0,\ell})=-\frac{1}{\tau_\ell^{2/\ell}}\; . \]
This means that point $x_{0,\ell}$ is an attracting orientation reversing fixed point
of $G_{\ell}$ and $G_{\ell}(x_{0,\ell})\to -1$ as $\ell\to \infty$.
\end{fact}
As in Lemma~\ref{lem:26ga,2}, we then get
that for any $\ell$ function
$\log\phi_{\ell}$ is integrable on the interval $J_\ell$.
This allows to repeat the proof of Lemma~\ref{lem:26ga,1}
which gives the identical {\em formula for the drift}: 
\begin{equation}\label{drforuni}
\theta(\ell)= -(\log|\tau_{\ell}|)^{-1} \int_{J_\ell}
\log\frac{|\phi_{\ell}(x)|}{|x|}\, f_{\ell}(x)dx\; .
\end{equation}
\subsection{Passing to the limit}
In order to show that $\theta(\ell)\to +\infty$
we prove that for any sequence $\ell_i\to \infty$ there is a subsequence
$\ell_{i(j)}$ such that $\theta(\ell_{i(j)})\to +\infty$.
For the proof it is enough that the sequence $(\phi_\ell)$ is pre-compact
(which is easier to show than the convergence).
\begin{lem}\label{compuni}
Given $\ell_i\to \infty$ there is a subsequence $\ell_{i(j)}$ such that
there exist limits $x_{0}=\lim x_{0,\ell_{i(j)}}\in (0,1)$, $X=\lim X_{\ell_{i(j)}}$, 
$\tau=\lim \tau_{\ell_{i(j)}}>1$ and
$\phi_{\ell_{i(j)}}$ converge on $(0,\tau x_0)$ uniformly on compacts to a function $\phi_\infty$
such that $\phi_\infty(x)\ge 0$, $\phi_\infty(x)=0$ if and only if $x=x_0$.
Moreover, the map $\phi_\infty$ is real analytic on $(0, x_0)\cup (x_0, \tau x_0)$.
Furthermore, the function $G_\infty=\tau^{-1} \phi_\infty \tau^{-1}$ is real analytic on $(0, \tau x_0)$,
$G_\infty(x_0)=x_0$, $G_\infty'(x_0)=-1$ and $G_\infty^2(x)=x-c (x-x_0)^3+O(|x-x_0|^4)$, where $c>0$.
Finally, 
\begin{equation}\label{guni}
\tau^{-2} \phi_\infty(x)=\phi_\infty G_\infty(x),
\end{equation}
for $x$ where both sides are defined.
\end{lem}
The proof is based on Fact~\ref{uni} and then proceeds as in~\cite{feig}, Theorem 2.
One takes into account that
$\phi_\ell(1/\tau_\ell)=1/\tau_\ell^2$, $\phi_\ell'(1/\tau_\ell)=1$ and 
$\phi(x_{0,\ell}/\tau_\ell)=\tau_\ell x_{0,\ell}$.

Lemma~\ref{compuni} allows us to define the induced map $\Phi_\infty: J_\infty\to J_\infty$
($J_\infty=(X, \tau X)$)
and the transfer operator $P_\infty$ using the limit map $\phi_\infty$ and then conclude 
that Corollary~\ref{corouni} and Fact~\ref{densuni} hold for $\ell=\infty$.
In particular, there is unique invariant density $f_\infty$ for $P_\infty$
which is real analytic and strictly positive on $J_\infty$.
Moreover, repeating the proof of Proposition~\ref{prop:24gp,1}, we have that
the sequence of densities $f_{\ell_{i(j)}}$ tend to $f_\infty$ almost uniformly on $J_\infty$.

On the other hand, by~(\ref{guni}), the function $-(\log\tau^4)^{-1}\log\phi_\infty$ is a Fatou coordinate
for $G_\infty^2$ at the fixed point $x_{0}$. Since $G_\infty^2$ has two attracting petals, we have 
that necessarily as $x\to x_0$, 
$$\log\phi_\infty(x)=C_0(x-x_0)^{-2}+C_1(x-x_0)^{-1}+O(|\log|x-x_0||),$$
where $C_0<0$.
There is $K>0$ such that $\log\phi_{\ell}(x)f_{\ell}(x)<K$ for all $\ell$ large or $\ell=\infty$
and all $x\in J_{\ell}$. Hence,
$$\int_{J_\infty}\log\phi_\infty(x)f_\infty(x)dx=-\infty.$$
Coupled with the fact that $\log\phi_{\ell_{i(j)}}(x)f_{\ell_{i(j)}}(x)\to \log\phi_\infty (x)f_\infty(x)$
and $J_{\ell_{i(j)}}\to J_\infty$, this implies that
$$\theta(\ell)=-(\log|\tau_{\ell}|)^{-1} \int_{J_\ell}
\log\frac{|\phi_{\ell}(x)|}{|x|}\, f_{\ell}(x)dx\to +\infty$$
along the subsequence $\ell_{i(j)}$.

\end{document}